\documentclass[11pt]{amsart}
\input xy
\xyoption{all}
\usepackage{amsmath,amsthm, amscd, amssymb, amsfonts}

\usepackage{tikz} 
\usepackage{tikz-cd}

\newcommand{\fpd}{\mbox{\rm FPdim\,}}

\newcommand{\comp}{\mbox{comp}}
\newcommand{\coev}{\mbox{coev}}
\newcommand{\ev}{\mbox{ev}}

\newcommand{\otk}{{\otimes_{\ku}}}

\newcommand{\Mo}{{\mathcal M}}
\newcommand{\No}{{\mathcal N}}

\newcommand{\Bimo}{{\mathcal Bimod}}

\newcommand{\uc}{{\mathcal U}}

\newcommand{\ca}{{\mathcal C}}

\newcommand{\ot}{{\otimes}}

\newcommand{\op}{\rm{op}}

\newcommand{\Ec}{{\mathcal E}}
\newcommand{\Ac}{{\mathcal A}}

\newcommand{\Zc}{{\mathcal Z}}

\newcommand{\cha}{{\mathcal A}}

\newcommand{\ele}{{\mathcal L}}

\newcommand{\A}{{\mathcal A}}

\newcommand{\Do}{{\mathcal D}}

\newcommand{\Bc}{{\mathcal B}}

\newcommand{\rev}{\rm{rev}}

\newcommand{\cop}{\rm{cop}}

\newcommand{\ra}{\rm{ra}}
\newcommand{\la}{\rm{la}}
\newcommand{\cf}{\rm{CF}}

\newcommand{\ku}{{\Bbbk}}

\newcommand{\uno}{ \mathbf{1}}
\newcommand{\C}{{\mathcal C}}

\newcommand{\id}{\mbox{\rm id\,}}

\newcommand{\Id}{\mbox{\rm Id\,}}

\newcommand{\Res}{\mbox{\rm Res\,}}

\newcommand{\vect}{\mbox{\rm vect\,}}

\newcommand{\Nat}{\mbox{\rm Nat\,}}
\newcommand{\Rex}{\mbox{\rm Rex\,}}
\newcommand{\Fun}{\operatorname{Fun}}
\newcommand\Rep{\operatorname{Rep}}

\newcommand\Hom{\operatorname{Hom}}
\newcommand\uhom{\underline{\Hom}}

\newcommand{\End}{\operatorname{End}}

\renewcommand{\_}[1]{\mbox{$_{\left( #1 \right)}$}}

\theoremstyle{plain}

\numberwithin{equation}{section}

\newtheorem{question}{Question}[section]

\newtheorem{teo}{Theorem}[section]

\newtheorem{lema}[teo]{Lemma}

\newtheorem{cor}[teo]{Corollary}

\newtheorem{prop}[teo]{Proposition}

\newtheorem{claim}{Claim}[section]

\theoremstyle{definition}

\newtheorem{defi}[teo]{Definition}

\theoremstyle{remark}

\newtheorem{rmk}[teo]{Remark}

\def\pf{\begin{proof}}

\def\epf{\end{proof}}

\theoremstyle{remark}

\begin{document}

\title[ Relative adjoint  algebras ]
{ Relative adjoint  algebras}
\author[     Mombelli  ]{ Mart\'in Mombelli
 }

\address{Facultad de Matem\'atica, Astronom\'\i a y F\'\i sica
\newline \indent
Universidad Nacional de C\'ordoba
\newline
\indent CIEM -- CONICET
\newline \indent Medina Allende s/n
\newline
\indent (5000) Ciudad Universitaria, C\'ordoba, Argentina}
\email{martin10090@gmail.com, martin.mombelli@unc.edu.ar
\newline \indent\emph{URL:}\/ https://www.famaf.unc.edu.ar/$\sim$mombelli
 }

\begin{abstract}  Given a finite tensor category $\ca$, an exact indecomposable $\ca$-module category $\Mo$, and a tensor subcategory $\Do\subseteq \ca^*_\Mo$, we describe a way to produce \textit{exact} commutative algebras in the center $Z(\ca)$, measuring this inclusion. The construction of such algebras is done in an analogous way as presented by Shimizu \cite{Sh2},  but using instead the \textit{relative (co)end}, a categorical tool developed in \cite{BM} in the realm of representations of tensor categories. We  provide some explicit computations.
\end{abstract}

\date{\today}
\maketitle


\section*{Introduction}

Assume $\Mo$ is a module category over a finite tensor category $\ca$.   In \cite{Sh2}, the author introduces the notion of \textit{adjoint algebra} $\cha_\Mo$ and the space of class functions $\cf(\Mo)$ associated to $\Mo$, generalizing the definitions given in \cite{Sh1}. This algebra is an interesting object associated   to  $\Mo$. In fact, this algebra is a complete invariant of $\Mo$.

\medbreak

If $\triangleright:\ca\times \Mo\to \Mo$ is the action of $\ca$ on $\Mo$, then we can consider  the  action functor 
$$\rho_\Mo:\ca\to \End(\Mo),$$ $$\rho_\Mo(X)(M)=X\triangleright M,\,\, X\in \ca, M\in\Mo.$$ The right adjoint of this functor is $\rho_\Mo^{\ra}:  \End(\Mo)\to\ca$, wich can be explicitly described as an end
$$ \rho_\Mo^{\ra}(F)= \int_{M\in \Mo} \uhom(M, F(M)),$$
for any  $F\in \End(\Mo)$, see \cite[Thm. 3.4]{Sh2}.  Here for any $M\in \Mo$, $\uhom(M, -) $ is the right adjoint of the functor $\ca\to \Mo$, $X\mapsto X\triangleright  M$. It is called the \textit{internal Hom }of the module category $\Mo$. The \textit{adjoint algebra} is defined as 
$$\cha_\Mo= \rho_\Mo^{\ra}(\Id_\Mo).$$ This object has structure of commutative algebra in the Drinfeld center $Z(\ca)$. A key ingredient of the construction is the end of the functor
$$\uhom(-, -):\Mo^{\op}\times \Mo\to \ca. $$

In this paper we investigate what kind of structure it is obtained, when we replace the usual end by the relative end, a new categorical tool introduced in \cite{BM}. 

Assume $\Ac$ is any category, and $S:\Mo^{\op}\times \Mo\to \Ac$ is a functor equipped with natural isomorphisms 
 $$\beta^X_{M,N}: S(M,X\triangleright N)\to S(X^*\triangleright M,N).$$
 We call this isomorphism a $\ca$-\textit{prebalancing} of $S$. Any additive $\ku$-linear functor $S$ posses a $\vect_\ku$-prebalancing. However, in the general case, the prebalancing is an extra structure of the functor $S$. The \textit{relative end} of $S$
is an object  $E \in \Ac$
 that comes with dinatural transformations $\pi_M: E \xrightarrow{ . .} S(M,M)$ such that the  equation
\begin{equation}\label{dwq2} S(\ev_X\triangleright \id_M,  \id_M) \pi_M= S(m_{X^*,X,M},  \id_M) \beta^X_{X\triangleright M,M} \pi_{X\triangleright M},\end{equation}
is fulfilled, and it is universal among all objects in $\Ac$ with dinatural transformations that satisfy \eqref{dwq2}. Unlike the case $\ca= \vect_\ku$, it may happen that a dinatural transformation does not satisfy  \eqref{dwq2}. The relative end of $S$ is denoted by $\oint_{M\in \Mo} (S,\beta)$. When $\ca= \vect_\ku$, the relative end coincides with the usual end. Relative coends are defined similarly.

If  $\Do\subseteq \ca^*_\Mo$ is a tensor subcategory, then $\Do$ acts on the left and on the right on $\End(\Mo)$ making it a $\Do$-bimodule category. The action functor 
$$\rho_\Mo:\ca\to \End(\Mo) $$
lands in $ Z_\Do(\End(\Mo))$, the relative center of  $\End(\Mo)$. Thus we can consider the functor
\begin{equation}\label{rho-01}
\rho:\ca\to  Z_\Do(\End(\Mo)).
\end{equation}  
Given any $F\in Z_\Do(\End(\Mo))$,  the internal Hom of the module category $\Mo$ posses a canonical $\Do$-prebalancing, see Lemma \ref{inthomfunct-preb}. Thus we can consider the relative end
$$\oint_{M\in \Mo} \uhom( - , F(-))\in \ca. $$

The \textit{adjoint algebra relative to} $\Do$ is then defined as
\begin{equation}\label{definition-a}
\Ac_{\Do,\Mo}:=\oint_{M\in \Mo} \uhom( - , -)\in \ca.
\end{equation} 
See Definition \ref{def-a1}. The algebra $ \Ac_{\Do,\Mo}$ seems to measure very well the inclusion $\Do\subseteq \ca^*_\Mo$. The object $ \Ac_{\Do,\Mo}$ is in fact a commutative connected algebra in the Drinfeld center $Z(\ca)$ such that $Z(\ca)_{\Ac_{\Do,\Mo}}$ is a rigid monoidal category. This last condition is equivalent to $Z(\ca)_{\Ac_{\Do,\Mo}}$ being an \textit{exact} module category over $Z(\ca)$, that is $\Ac_{\Do,\Mo}$ is an \textit{exact} algebra. We called this type of algebras \textit{étale}, generalizing the notion of étale algebras in the semisimple case. See \cite[Section 3]{DMNO}. In our case, separability of the algebra is replaced by the condition of exactness. 
\medbreak

The key step to prove all these facts, is the computation of a right adjoint of the functor $\rho$, given in \eqref{rho-01}. A right adjoint is described as 
$$ \bar{\rho}: Z_\Do(\End(\Mo))\to \ca,$$
$$\bar{\rho}(F)= \oint_{M\in \Mo} \uhom(M, F(M)).$$
In Theorem \ref{adjoint-structure} we prove this fact, and we also prove that $\bar{\rho}$ is exact, faithful and dominant. This implies that the adjunction $(\rho, \bar{\rho})$ is monadic, and this fact implies that there is a monoidal equivalence
$$Z(\ca)_{A_{\Do,\Mo}} \simeq Z_\Do(\ca^*_\Mo).$$
Hence we can compute its Frobenius-Perron dimension as $$\fpd(\Ac_{\Do, \Mo})= \frac{\fpd(\ca)}{\fpd(\Do)}.$$
These assertions are stated in Proposition \ref{et-1}.

The correspondence between tensor subcategories of $ \ca^*_\Mo$ and certain commutative algebras in the center is not new, at least in the semisimple case. In \cite[Thm. 4.10]{DMNO},  the authors described a bijective correspondence between tensor subcategories of a fixed tensor category $\ca$ and étale subalgebras of $\Ac_\ca.$ This subalgebra is also described as $I(\uno)$, where $I$ is the right adjoint of a certain functor. What we present new in this paper is an explicit description of the functor $I$, and the means to compute the algebra $\Ac_{\Do, \Mo}$.

\medbreak

One of the advantages of the description presented in \eqref{definition-a}, is that it provides a formula for its computation. In Section \ref{Section:comp} we present explicit calculations. We provide examples when $\ca$ is an equivariantization by a finite group, when $\ca$ is the tensor category of representations of a finite-dimensional Hopf algebra, and when  $\ca$ is the Deligne tensor product of two tensor categories. 
\subsection*{Acknowledgments} This  work  was  partially supported by CONICET and Secyt (UNC),  Argentina.

\section{Preliminaries}

Throughout this paper, $\ku$ will denote an algebraically closed field of characteristic zero. All categories will be finite, in the sense of \cite{EO}, abelian $\ku$-linear categories, and all functors will be additive $\ku$-linear. If $F:\ca\to \Do$ is a functor, we shall denote by $F^{\la}$, $F^{\ra}$ the left and right adjoint to $F$, if they exist.

\subsection{Monads} A monad over a category $\ca$ is an algebra object in the monoidal category  $\End(\ca)$. That is, a functor $T:\ca\to \ca$ equipped with natural morphisms
$$ \mu:T\circ T\to T, \, \, \, \eta:\Id\to T,$$
the multiplication and the unit of $T$, such that they satisfy associativity and unity axioms. The category of $T$-modules, $\ca^T$ is the category whose objects are pairs $(M,s)$, where $M$ is an object of $\ca$, $s:T(M)\to M$ is a morphism in $\ca$ such that
$$s T(s)= s \mu_M, \quad s \eta_M= \id_M. $$
The forgetful functor $\uc:\ca^T\to \ca$ has left adjoint given by the \textit{free} functor $\ele:\ca\to \ca^T$, $\ele(M)=(T(M), \mu_M). $

\bigbreak

If $(G:\Do\to \ca, F:\ca\to \Do)$ is an adjunction, with unit and counit $e:G\circ F\to \Id, $ $c:\Id\to F\circ G$, $ $ then $T=F\circ G$ is a monad. The \textit{comparison functor}
$$ \kappa: \Do\to \ca^T,$$
is defined by $\kappa(M)=(F(M), F(e_M)), $ for any $M\in \Do$. If the comparison functor $\kappa$ is an equivalence of categories, the adjunction  $(G,F)$ is said to be \textit{monadic}. Several results  concerning monadicity of an adjunction are given by Beck's Theorem. See for example \cite{M}. As a consequence of Beck's Monadicity Theorem, it follows that, if $F:\ca\to \Do$ is faithful and exact, then $(G,F)$ is monadic.

\subsection{Finite tensor categories}

A  tensor category $\ca$ is a monoidal rigid category, with simple unit object $\uno\in \ca$, satisfying certain finitness condition. We refer to \cite{EO} for more details on finite tensor categories.

 We shall denote by $\ca^{\rev}$, the tensor category whose underlying abelian category is $\ca$, with \textit{reverse} monoidal product 
    $$\ot^{\rev}:\ca\times \ca\to \ca, \quad X\ot^{\rev }Y=Y\ot X,$$ and associativity constraints
 $$a^{\rev}_{X,Y,Z}:( X\ot^{\rev }Y) \ot^{\rev} Z\to  X\ot^{\rev }(Y \ot^{\rev} Z),$$
  $$a^{\rev}_{X,Y,Z}:= a^{-1}_{Z,Y,X},$$
  for any $X,Y, Z\in \ca$. Here  $a_{X,Y,Z}:(X\ot Y)\ot Z\to X\ot (Y\ot Z) $ is the associativity of $\ca$.

\medbreak

The next result is \cite[Prop. 5.1]{BN}. It will be useful later.
\begin{prop}\label{dominant-monads} Let $\ca, \Do$ be tensor categories, and let $F:\ca\to \Do$ be an exact tensor functor with left adjoint $G:\Do\to \ca$.  Let $T= FG$ be the monad associated to $F$.
 Let $T = FG$ be the Hopf monad associated to the adjunction $(G,F)$. Then the following assertions are equivalent:
 \begin{itemize}
\item[(i) ] The functor F is dominant; 
 \item[(ii) ] The unit $\eta$ of the monad T is a monomorphism;
 \item[(iii) ] The monad $T$ is faithful;
 
 \item[(iv) ] The left adjoint of F is faithful;
 \item[(v) ] The right adjoint of F is faithful.\qed
 \end{itemize}
\end{prop} 

If $A\in \ca$ is an algebra, then $\ca_A$, respectively ${}_A\ca_A$, denote the category of right $A$-modules in $\ca$, respectively the category of $A$-bimodules in $A$. We shall denote by 
$\uc:\ca_A\to \ca,$
the \textit{forgetful functor}, and a left adjoint $\ele_A:\ca\to \ca_A$, $\ele_A(V)=V\ot A$, the \textit{free functor}. A \textit{central algebra} for $\ca$ is an algebra in the Drinfeld center  $Z(\ca)$; that is a pair $(A,\sigma)$, where $\sigma$ is a half-braiding for $A$. It is commutative if $m\sigma = m$, where $m:A\ot A\to A$ is the multiplication of $A$.

\medbreak

The category ${}_A\ca_A$ has a monoidal product $\ot_A$. For this category to be rigid, we need to ask for more conditions on the algebra $A$. If $(A,\sigma)$ is a commutative central algebra, then there is a faithful functor $$\ca_A \to {}_A\ca_A,$$
described as follows. If $V\in \ca_A$, then we can define a left $A$-module structure using the half-braiding of $A$:
$$A\ot V\xrightarrow{ \sigma_V } V\ot A\to V. $$
The category $\ca_A$ inherits from ${}_A\ca_A$ a structure of  monoidal category with unit object $A$, such that the free functor $\ele_A:\ca\to \ca_A$ is a (strong) monoidal functor. Observe that both, the monoidal structure of $\ca_A$ and the monoidal structure of $\ele$ depend on the half-braiding $\sigma$.

The next result is \cite[Prop. 6.1]{BN}, and it will be used later. 
\begin{prop}\label{central-adj-alg} Let $\ca, \Do$ be tensor categories, and let $F:\ca\to \Do$ be an exact dominant tensor functor with right adjoint $R:\Do\to \ca$. The following statements hold.
\begin{itemize}
\item[1.] The object $R(\uno)=A$ has a natural structure of central commutative algebra in $\ca$.

\item[2.] If $R$ is exact and faithful, then $\ca_A$ is a tensor category, and there exists a tensor equivalence $K:\Do\to \ca_A$ such that $K\circ F\simeq \ele_A$ as tensor functors.\qed

\end{itemize}
\end{prop}

 \section{Representations of tensor categories} A  left \emph{module} category over  
$\ca$ is a  category $\Mo$ together with a 
bifunctor $\rhd: \ca \times \Mo \to \Mo$, exact in each variable,  endowed with 
 natural associativity
and unit isomorphisms 
$$m_{X,Y,M}: (X\otimes Y)\triangleright   M \to X\triangleright  
(Y\triangleright M), \ \ \ell_M: \uno \triangleright  M\to M.$$ 
These isomorphisms are subject to the following conditions:
\begin{equation}\label{left-modulecat1} m_{X, Y, Z\triangleright M}\; m_{X\otimes Y, Z,
M}= (\id_{X}\triangleright m_{Y,Z, M})\;  m_{X, Y\otimes Z, M}(a_{X,Y,Z}\triangleright\id_M),
\end{equation}
\begin{equation}\label{left-modulecat2} (\id_{X}\triangleright \ell_M)m_{X,{\bf
1} ,M}= r_X \triangleright \id_M,
\end{equation} for any $X, Y, Z\in\C, M\in\Mo.$ Here $a$ is the associativity constraint of $\C$.
Sometimes we shall also say  that $\Mo$ is a (left) $\ca$-\emph{module category} or a (left) representation of $\ca$.

\medbreak

Let $\Mo$ and $\Mo'$ be a pair of $\C$-modules. A\emph{ module functor} is a pair $(F,c)$, where  $F:\Mo\to\Mo'$  is a functor equipped with natural isomorphisms
$$c_{X,M}: F(X\triangleright M)\to
X\triangleright F(M),$$ $X\in  \ca$, $M\in \Mo$,  such that
for any $X, Y\in
\ca$, $M\in \Mo$:
\begin{align}\label{modfunctor1}
(\id_X \triangleright  c_{Y,M})c_{X,Y\triangleright M}F(m_{X,Y,M}) &=
m_{X,Y,F(M)}\, c_{X\otimes Y,M}
\\\label{modfunctor2}
\ell_{F(M)} \,c_{\uno ,M} &=F(\ell_{M}).
\end{align}

There is a composition
of module functors: if $\Mo''$ is a $\C$-module category and
$(G,d): \Mo' \to \Mo''$ is another module functor then the
composition
\begin{equation}\label{modfunctor-comp}
(G\circ F, e): \Mo \to \Mo'', \qquad  e_{X,M} = d_{X,F(M)}\circ
G(c_{X,M}),
\end{equation} is
also a module functor.

\smallbreak  

A \textit{natural module transformation} between  module functors $(F,c)$ and $(G,d)$ is a 
 natural transformation $\theta: F \to G$ such
that
\begin{gather}
\label{modfunctor3} d_{X,M}\theta_{X\triangleright M} =
(\id_{X}\triangleright \theta_{M})c_{X,M},
\end{gather}
 for any $X\in \ca$, $M\in \Mo$. The vector space of natural \textit{module} transformations will be denoted by $\Nat_{\!m}(F,G)$. Two module functors $F, G$ are \emph{equivalent} if there exists a natural module isomorphism
$\theta:F \to G$. We denote by $\Fun_{\ca}(\Mo, \Mo')$ the category whose
objects are module functors $(F, c)$ from $\Mo$ to $\Mo'$ and arrows module natural transformations. 

\medbreak
Two $\C$-module categories $\Mo$ and $\Mo'$ are {\em equivalent} if there exist module functors $F:\Mo\to
\Mo'$, $G:\Mo'\to \Mo$, and natural module isomorphisms
$\Id_{\Mo'} \to F\circ G$, $\Id_{\Mo} \to G\circ F$.
\medbreak
A module is
{\em indecomposable} if it is not equivalent to a direct sum of
two non trivial modules. Recall from \cite{EO}, that  a
module $\Mo$ is \emph{exact} if   for any
projective object
$P\in \ca$ the object $P\triangleright M$ is projective in $\Mo$, for all
$M\in\Mo$. If $\Mo$ is an exact indecomposable module category over $\ca$, the dual category $\ca^*_\Mo=\End_\ca(\Mo)$ is a finite tensor category \cite{EO}. The tensor product is the composition of module functors.

\medbreak
 A \emph{right module category} over $\ca$
 is a finite  category $\Mo$ equipped with an exact
bifunctor $\triangleleft:  \Mo\times  \ca\to \Mo$ and natural   isomorphisms 
$$\widetilde{m}_{M, X,Y}: M\triangleleft (X\ot Y)\to (M\triangleleft X) \triangleleft Y, \quad r_M:M\triangleleft \uno\to M$$ such that
\begin{equation}\label{right-modulecat1} \widetilde{m}_{M\triangleleft X, Y ,Z }\; \widetilde{m}_{M,X ,Y\ot Z } (\id_M \triangleleft a_{X,Y,Z})=
(\widetilde{m}_{M,X , Y}\triangleleft \id_Z)\, \widetilde{m}_{M,X\ot Y ,Z },
\end{equation}
\begin{equation}\label{right-modulecat2} (r_M\triangleleft \id_X)  \widetilde{m}_{M,\uno, X}= \id_M\triangleleft l_X.
\end{equation}

If $\Mo,  \Mo'$ are right $\ca$-modules, a module functor from $\Mo$ to $  \Mo'$ is a pair $(T, d)$ where
$T:\Mo \to \Mo'$ is a  functor and $d_{M,X}:T(M\triangleleft X)\to T(M)\triangleleft X$ are natural  isomorphisms
such that for any $X, Y\in
\ca$, $M\in \Mo$:
\begin{align}\label{modfunctor11}
( d_{M,X}\ot \id_Y)d_{M\triangleleft X, Y}T(m_{M, X, Y}) &=
m_{T(M), X,Y}\, d_{M, X\ot Y},
\\\label{modfunctor21}
r_{T(M)} \,d_{ M,\uno} &=T(r_{M}).
\end{align}

We recall the next well-known result since it will be needed later. The proof it is done by a straightforward calculation.
Let $\Mo, \Mo'$ be left $\ca$-module categories. Assume that $F:\Mo\to\Mo'$ is a functor with left adjoint $F^{\la}:\Mo'\to\Mo$, with associated natural isomorphisms
$$ \alpha_{M,N}:\Hom_{\Mo}(F^{\la}(M),N)\to \Hom_{\Mo'}(M,F(N)).$$
\begin{lema} Let $e: F^{\la}\circ F\to \Id$, $\eta:\Id\to F\circ F^{\la}$ be the unit and counit of the adjunction. Also assume that $(F,c):\Mo\to\Mo'$ is a module functor. Then $F^{\la}$ has structure of module functor, with structure given by
$$\widetilde{c}_{X,M}:F^{\la}(X\triangleright M)\to X\triangleright F^{\la}(M), $$
$$ \widetilde{c}_{X,M}= e_{X\triangleright F^{\la}(M)} F^{\la}(c^{-1}_{X, F^{\la}(M)}) F^{\la}(\id_X \triangleright \eta_M),$$
for any $X\in \ca$, $M\in \Mo'.$ Aslo, equation
\begin{equation}\label{adjointp-mod-funct}
\alpha^{-1}_{X\triangleright F(M), X\triangleright M}(c^{-1}_{X,M})= (\id_X \triangleright e_M) \widetilde{c}_{X,F(M)}.
\end{equation}
is fulfilled, for any $X\in \ca,$ $M\in \Mo'.$\qed
\end{lema}

\subsection{\'Etale algebras in finite tensor categories}\label{section:etale} 

We propose the following definition of \'etale algebras in the non-semisimple case, generalizing \cite[Definition 3.1]{DMNO}. Assume that $\ca$ is a braided tensor category.

\begin{defi}\label{definition:etale} An algebra $A\in \ca$ is \textit{\'etale} if:
\begin{itemize}
\item[$\bullet$] $A$  is commutative;

\item[$\bullet$] the category $\ca_A$ is an exact indecomposable left module category over $\ca$;

\item[$\bullet$] $A$ is simple as a right $A$-module; that is $A\in \ca_A$ is a simple object.

\end{itemize}  It is \textit{connected} if $\dim_\ku \Hom_\ca(\uno, A)=1$.
\end{defi}

\begin{rmk}  The exactness and indecomposability of $\ca_A$ as a left $\ca$-module category  may be a superfluous condition. We expect that $A$ being right simple implies that   $\ca_A$ is  an exact and indecomposable  module category.
\end{rmk}

\begin{rmk} Condition of $\ca_A$ being exact and indecomposable  module category is equivalent to ${}_A\ca_A$ being a rigid monoidal category. Indeed, there is a monoidal equivalence $ {}_A\ca_A\simeq \ca^*_{\ca_A}$. One can prove that this is also equivalent to $\ca_A$ being a rigid monoidal category, since rigidity of objects in $\ca_A$ is iherited by rigidity in ${}_A\ca_A$.
\end{rmk}

There is a canonical correspondence between \'etale algebras and dominant tensor functors with an exact faithful right adjoint.  The proof of the next result follows from Proposition \ref{central-adj-alg}. See also \cite[Lemma 3.5]{DMNO}.

\begin{cor}\label{from adj-to-etale} If $F:Z(\ca)\to \Do$ is an exact dominant tensor functor with an exact faithful right adjoint $R:\Do\to Z(\ca)$, then $(A,\sigma)=R(\uno)$ is a connected \'etale algebra in the center $Z(\ca)$.\qed
\end{cor}

\subsection{The internal Hom}\label{subsection:internal hom} Let $\ca$ be a  tensor category and $\Mo$ be  a left $\C$-module category. For any pair of objects $M, N\in\Mo$, the \emph{internal Hom} is an object $\uhom(M,N)\in \C$ representing the left exact functor $$\Hom_{\Mo}(-\triangleright M,N):\ca^{\op}\to \vect_\ku.$$ This means that, there are natural isomorphisms, one the inverse of each other,
\begin{equation}\label{Hom-interno}\begin{split}\phi^X_{M,N}:\Hom_{\ca}(X,\uhom(M,N))\to \Hom_{\Mo}(X\triangleright M,N), \\
\psi^X_{M,N}:\Hom_{\Mo}(X\triangleright M,N)\to \Hom_{\ca}(X,\uhom(M,N)),
\end{split}
\end{equation}
 for all $M, N\in \Mo$, $X\in\ca$. 
Sometimes we shall denote the internal Hom of the module category $\Mo$ by $\uhom_\Mo$ to emphasize that it is related to $\Mo$.

Naturality of $\psi$ implies that diagrams
\begin{equation}\label{comm-psi1}
\xymatrix{
\Hom_{\Mo}(X\triangleright M,N)\ar[d]^{\beta \mapsto f\circ\beta}\ar[rr]^{\psi^X_{M,N}}&& \Hom_{\ca}(X,\uhom(M,N)) \ar[d]_{\alpha\mapsto \uhom(M,f)\circ\alpha} \\
\Hom_{\Mo}(X\triangleright M,\widetilde{N}) \ar[rr]^{\psi^X_{M,\widetilde{N}}}&& \Hom_{\ca}(X,\uhom(M,\widetilde{N})), }
\end{equation}
\begin{equation}\label{comm-psi2}
\xymatrix{
\Hom_{\Mo}(X\triangleright \widetilde{M},N)\ar[d]^{\beta \mapsto \beta(\id_X\triangleright g) }\ar[rr]^{\psi^X_{\widetilde{M},N}}&& \Hom_{\ca}(X,\uhom(\widetilde{M},N)) \ar[d]_{\alpha\mapsto \uhom(g,N)\circ \alpha} \\
\Hom_{\Mo}(X\triangleright M,N) \ar[rr]^{\psi^X_{M, N}}&& \Hom_{\ca}(X,\uhom(M,N)), }
\end{equation}
\begin{equation}\label{comm-psi3}
\xymatrix{
\Hom_{\Mo}(Y\triangleright M,N)\ar[d]^{\beta \mapsto \beta(h\triangleright \id_M) }\ar[rr]^{\psi^Y_{M,N}}&& \Hom_{\ca}(Y,\uhom(M,N)) \ar[d]_{\alpha\mapsto  \alpha\circ h} \\
\Hom_{\Mo}(X\triangleright M,N) \ar[rr]^{\psi^X_{M, N}}&& \Hom_{\ca}(X,\uhom(M,N)), }
\end{equation}
commute, for any pair of morphism $f:N\to \widetilde{N}$, $g:M\to \widetilde{M}$  in $\Mo$, and any morphism $h:X\to Y$ in $\ca$. Also, naturality of $\phi$, implies that  diagrams
\begin{equation}\label{comm-phi1}
\xymatrix{
\Hom_{\ca}(X,\uhom(\widetilde{M},N))\ar[d]^{\beta \mapsto \uhom(g,N)\circ\beta}\ar[rr]^{\phi^X_{\widetilde{M},N}}&& \Hom_{\Mo}(X\triangleright \widetilde{M},N) \ar[d]_{\alpha\mapsto \alpha (\id_X\triangleright g)} \\
\Hom_{\ca}(X,\uhom(M,N)) \ar[rr]^{\phi^X_{M,N}}&& \Hom_{\Mo}(X\triangleright M,N), }
\end{equation}
\begin{equation}\label{comm-phi2}
\xymatrix{
\Hom_{\ca}(X,\uhom(M,N))\ar[d]^{\beta \mapsto \uhom(M,f)\beta}\ar[rr]^{\phi^X_{M,N}}&& \Hom_{\Mo}(X\triangleright M,N) \ar[d]_{\alpha\mapsto f\alpha} \\
\Hom_{\ca}(X,\uhom(M,\widetilde{N})) \ar[rr]^{\phi^X_{M,\widetilde{N}}}&& \Hom_{\Mo}(X\triangleright M,\widetilde{N}), }
\end{equation}
\begin{equation}\label{comm-phi3}
\xymatrix{
\Hom_{\ca}(Y,\uhom(M,N))\ar[d]^{\beta \mapsto \beta\circ h}\ar[rr]^{\phi^Y_{M,N}}&& \Hom_{\Mo}(Y\triangleright M,N) \ar[d]_{\alpha\mapsto \alpha (h\triangleright\id_M)} \\
\Hom_{\ca}(X,\uhom(M,N)) \ar[rr]^{\phi^X_{M,N}}&& \Hom_{\Mo}(X\triangleright M,N), }
\end{equation}

 commutes for any $f:N\to  \widetilde{N}$, $g:M\to\widetilde{M}$ morphisms in $\Mo$, and any morphism $h:X\to Y$ in $\ca$. The next result will be needed later. 
\begin{lema}\label{int-hom-adjoint}\cite[Lemma 3]{FS} Let $\Mo$ be an  exact  module category over $\ca$, and $F:\Mo\to \Mo$ be a $\ca$-module functor with left adjoint $F^{\la}:\Mo\to \Mo$. Then, there are natural isomorphisms
$$ \xi^F_{M,N}:\uhom(M, F(N))\to  \uhom(F^{\la}(M), N).$$\qed
\end{lema}

\begin{rmk} Isomorphisms $\xi^F$ can be explicitly described using the proof of \cite[Lemma 3]{FS}. One can see that 
\begin{equation}\label{xi-definition} 
\xi^F_{M, N}= \psi^{Y}_{F^{\la}(M),N}\big(\alpha^{-1}_{Y\triangleright M, N}(\phi^Y_{M,F(N)}(\id)) \tilde{d}^{-1}_{Y,M}\big).
\end{equation}
Here $Y= \uhom(M, F(N))$,  $\tilde{d}$ is the module structure of the functor $ F^{\la}$, and
$$\alpha_{M,N}:\Hom_\Mo(F^{\la}(M),N)\to \Hom_\Mo( M,F(N))$$
are the natural isomorphisms of the adjunction $(F^{\la},F).$
\end{rmk}

Next, we shall define a collection of morphisms that will be used continuosly. If $X\in \ca$, $M, N\in \Mo$, set 
$$\coev^\Mo_{ M}:\uno\to  \uhom(M, M), \quad \ev^\Mo_{M, N}:\uhom(M,N)\triangleright M\to N,$$
$$\coev^\Mo_{ M}=\psi^\uno_{M, M}(\id_{ M}),  \quad\ev^\Mo_{M, N}= \phi^{\uhom(M,N)}_{M,N}(\id_{\uhom(M,N)}).$$

Define also $f_M=\ev^\Mo_{M,M}(\id_{\uhom(M,M)}\triangleright \ev^\Mo_{M, M})$, and
\begin{equation}\label{compositionM}
\comp^\Mo_M:\uhom(M,M)\ot \uhom(M,M)\to \uhom(M,M),
\end{equation}
$$\comp^\Mo_M=\psi^{\uhom(M,M)\ot \uhom(M,M)}_{M,M}(f_M). $$

For any $M\in \Mo$ denote by $R_M:\ca\to \Mo$ the   $\ca$-module functor $R_M(X)=X\triangleright M$. Its right adjoint $R^{ra}_M: \Mo\to \ca$ is then $ R^{\ra}_M(N)=\uhom(M,N)$. The functor  $ R^{\ra}_M$ is also a $\ca$-module functor.  Denote by
\begin{equation}\label{mod-struct-a}
  \mathfrak{a}_{X,M,N}: \uhom(M, X\triangleright  N) \to X\ot \uhom(M,  N)
\end{equation}
the left $\ca$-module structure of $R^{\ra}_M$. Let $ \mathfrak{b}^1_{X,M,N}:\uhom(X\triangleright  M, N)\to \uhom(M,N)\ot X^* $  
be the isomorphisms induced by the natural isomorphisms
\begin{gather*}
  \Hom_\ca(Z, \uhom(X \triangleright M, N))
  \simeq \Hom_{\mathcal{M}}(Z \triangleright( X \triangleright M), N)\\ \simeq \Hom_{\mathcal{M}}((Z \ot X)\triangleright M, N)
  \simeq \Hom_{\mathcal{C}}(Z \otimes X, \uhom(M, N))\\
  \simeq \Hom_{\mathcal{C}}(Z, \uhom(M, N) \otimes X^*),
\end{gather*}
for any $X, Z\in \ca$, $M, N\in \Mo$. Define also
$$ \mathfrak{b}_{X,M,N}:\uhom(X\triangleright M, N)\ot X\to \uhom(M,N)$$ 
as the composition
$$\mathfrak{b}_{X,M,N}=(\id\ot \ev_X)  (\mathfrak{b}^1_{X,M,N}\ot \id_X).$$
\subsection{The canonical action of $\vect_\ku$ }\label{Section:vect-action}

Any abelian $\ku$-linear category is a $\vect_\ku$-module category. If $\Mo$ is such a category, there exists a canonical action of $\vect_\ku$ on $\Mo$
$$ \blacktriangleright: \vect_\ku\times \Mo\to \Mo,$$
determined by
$$\Hom_\ku(V,\Hom_\Mo(A,B))\simeq \Hom_\Mo(V\blacktriangleright A,B) $$
for any $V\in \vect_\ku$, $A, B\in \Mo$. Any additive $\ku$-linear functor, results a $\vect_\ku$-module functor. 

\medbreak

There exists a compatibility of the action of $\vect_\ku$ and the action of any tensor category.
\begin{lema}\label{compati-actions} If $\ca$ is a tensor category and $\Mo$ is a left $\ca$-module category with action given by
$\rhd: \ca \times \Mo \to \Mo$, then there are natural isomorphisms
$$ (V \blacktriangleright \uno) \rhd M\simeq V \blacktriangleright  M,$$
for any $V\in \vect_\ku$, $X\in \ca$, $M\in \Mo$. \qed
\end{lema}

\section{Bimodule categories and the relative center}

Assume that $\ca, \Do, \Ec$ are   tensor categories. A $(\ca, \Do)-$\emph{bimodule category}  is a category $\Mo$  with left $\ca$-module category structure 
$\triangleright: \ca\times \Mo\to \Mo$,
 and right $\Do$-module category  structure $ \triangleleft: \Mo\times \Do \to \Mo$,
equipped with natural
isomorphisms $$\{\delta_{X,M,Y}:(X\triangleright M) \triangleleft Y\to X\triangleright  (M\triangleleft Y), X\in\ca, Y\in\Do, M\in \Mo\}$$ satisfying 
certain axioms. For details the reader is referred to \cite{Gr}, \cite{Gr2}. 

Assume that $\Mo$ is a $\ca$-bimodule category, with right and left  associativities given by

$$m^r_{M, X,Y}: M\triangleleft (X\ot Y)\to (M\triangleleft X) \triangleleft Y,$$
$$m^l_{X,Y,M}: (X\otimes Y)\triangleright   M \to X\triangleright  
(Y\triangleright M).$$

The \textit{ relative center} \cite{GNN} of $\Mo$ is the category $Z_\ca(\Mo)$ that consists of pairs $(M, \gamma)$, where $M\in \Mo$ and $\gamma_X: M \triangleleft X \to X \triangleright M$ is a family of natural isomorphism, called the \textit{half-braiding}, such that

\begin{equation}\label{relativecent1} m^l_{X,Y,M} \gamma_{X\ot Y}=  (\id_X \triangleright \gamma_Y) \delta_{X,M,Y} (\gamma_X\triangleleft \id_Y) m^r_{M, X,Y},
\end{equation}

for any $X, Y\in \ca$. If $(M, \gamma)$, $(M', \gamma')$ are two elements in $Z_\ca(\Mo)$, a morphism between them is a morphism $f:M\to M'$ such that

\begin{equation}\label{relativecent2} (\id_X \triangleright f) \gamma_X= \gamma'_X (f \triangleleft \id_X),
\end{equation}
for any $X\in \ca$.
The next result is \cite[Corollary B.5]{FSS1}.
\begin{prop}\label{adjoint-forgetf} The forgetful functor $\uc: Z_\ca(\Mo)\to \Mo,$ $\uc(M, \gamma)=M$, for any $(M, \gamma)\in Z_\ca(\Mo)$ is exact, and has left adjoint $I:\Mo\to Z_\ca(\Mo)$ given by
$$I(M)=\int^{X\in \ca} X\triangleright M \triangleleft {}^*X, $$
for any $M\in \Mo$.\qed
\end{prop}

The relative center can be thought of as a 2-functor
$$\Zc_\ca: {}_\ca\Bimo \to Ab_\ku$$
Here ${}_\ca\Bimo$ is the 2-category of $\ca$-bimodule categories, and $Ab_\ku$ is the 2-category of finite $\ku$-linear abelian categories. If $\Mo, \No$ are $\ca$-bimodule categories, and $(F,c,d):\Mo\to \No$ is a bimodule functor, then $$\Zc_\ca(F):\Zc_\ca(\Mo)\to \Zc_\ca(\No)$$ is the functor given as follows. If $(M,\gamma)\in\Zc_\ca(\Mo),$ then $\Zc_\ca(F)(M,\gamma)=(F(M), \widetilde\gamma)$, where $\widetilde\gamma_X: F(M)\triangleleft X\to X\triangleright F(M)$ is defined as
\begin{equation}\label{relative-bimod-funct} 
\widetilde\gamma_X=c_{X,M} F(\gamma_X)d^{-1}_{M,X},
\end{equation}
for any $X\in \ca$. The next result seem to be well-known.
\begin{prop}\label{about-Z} Let $\Mo, \No$ be $\ca$-bimodule categories, and $(F,c,d):\Mo\to \No$ be a bimodule functor. The following statements hold.
\begin{itemize}
\item[1.] If $G:\No\to \Mo$ is a bimodule functor, right adjoint to $F$, then $\Zc_\ca(G)$ is right adjoint to $\Zc_\ca(F)$.

\item[2.] If $F$ is exact (faithful) then $\Zc_\ca(F)$ is exact (respect.  faithful).\qed

\end{itemize}

\end{prop}

\section{ Adjoint algebras associated to module categories  }\label{Section:adjoints}

 Let $\ca$ be a  tensor category, and let $\Mo$ be an exact indecomposable left $\ca$-module category. In this Section, we shall recall the definition of the adjoint algebra associated to $\Mo$, as introduced by K. Shimizu \cite{Sh2}. \medbreak

The \textit{adjoint algebra} of the module category $\Mo$ was defined as follows. If
 $$\rho_\Mo:\ca\to \Rex(\Mo),$$ $$\rho_\Mo(X)(M)=X\triangleright M,  X\in \ca, M\in \Mo,$$ is the action functor, since it is a $\ca$-bimodule functor, then its right adjoint is also a $\ca$-bimodule action. Whence we can consider the functor
 $$Z_\ca(\rho_\Mo^{\ra}): \ca^*_\Mo=Z_\ca(\Rex(\Mo))\to Z(\ca). $$
 Then $\cha_\Mo:=Z_\ca(\rho_\Mo^{\ra})(\Id_\Mo)\in Z(\ca)$. The \textit{adjoint algebra of the tensor category }$\ca$ is the algebra $\cha_\ca$ of the regular module category $\ca$.

In \cite{Sh1}, \cite{Sh2} the author describes the right adjoint $\rho_\Mo^{\ra}$ using the end of the internal hom; that is, it was proven that 
$$\rho_\Mo^{\ra}:  \Rex(\Mo)\to\ca,$$ 
$$ \rho_\Mo^{\ra}(F)= \int_{M\in \Mo} \uhom(M, F(M)),$$
for any $F\in \Rex(\Mo)$. This description has several advantages, one of them is that it is possible to describe the algebra structure and the half braiding of $\cha_\Mo$ in terms of the dinatural transformations of the end.

Assume that $\pi^\Mo_M:\cha_\Mo\xrightarrow{ .. } \uhom(M, M)$ are the corresponding dinatural transformations. The half braiding of $\cha_\Mo$ is the unique isomorphism $\sigma^\Mo_X: \cha_\Mo\ot X\to X\ot \cha_\Mo$ such that the diagram
 \begin{equation}\label{half-braidig-ch}
    \xymatrix@C=90pt@R=16pt{
      \cha_\Mo \otimes X
      \ar[r]^-{\pi^\Mo_{X \triangleright M} \otimes \id_X}
      \ar[dd]_{\sigma^{\Mo}_X}
      & \uhom(X \triangleright M,  X \triangleright  M) \otimes X
      \ar[d]^{\mathfrak{b}_{X, M, X \triangleright  M}} \\
      & \uhom(M,  X \triangleright  M)
      \ar[d]^{ \mathfrak{a}_{X,M,M}} \\
      X \otimes \cha_\Mo
      \ar[r]^-{\id_X \otimes \pi^\Mo_M}
      & X \ot \uhom(M, M)
    }
  \end{equation}
is commutative. Recall from Section \ref{subsection:internal hom} the definition of the morphisms $\mathfrak{b}, \mathfrak{a}$.
It was explained in \cite[Subection 4.2]{Sh2} that the algebra structure of $\cha_\Mo$ is given as follows.  The product and the unit of $\cha_\Mo$ are 
$$m_\Mo:\cha_\Mo\ot \cha_\Mo\to \cha_\Mo, \quad u_\Mo:\uno\to \cha_\Mo,$$
defined to be the unique morphisms such that they satisfy
\begin{equation}\label{product-cha-M}
\begin{split}\pi^\Mo_M\circ m_\Mo= \comp^\Mo_{M}\circ (\pi^\Mo_M\ot \pi^\Mo_M),\\ \pi^\Mo_M\circ  u_\Mo= \coev^\Mo_{ M},\end{split}
\end{equation}

for any $M\in \Mo$. For the definition of $\coev^\Mo$ and $ \comp^\Mo$ see Section \ref{subsection:internal hom}.

\section{ (Co)ends for module categories}

In this Section we recall the definition of \textit{relative} (co)ends. This is a useful tool in the theory of representations of tensor categories, generalizing the usual (co)ends, that we developed in \cite{BM}.

Let $\ca$ be a  tensor category and $\Mo$ be a left $\ca$-module category. Let $m$ be the associativity constraint of $\Mo$.  Assume that $\Ac$ is a category together with a functor $S:\Mo^{\op}\times \Mo\to \Ac$, equipped with natural isomorphisms
\begin{equation*} \beta^X_{M,N}: S(M,X\triangleright N)\to S(X^*\triangleright M,N),
\end{equation*}
for any $X\in \ca, M,N\in \Mo$. Isomorphisms $\beta$ are called a \textit{prebalancing} of the functor $S$. Sometimes we shall refer to $\beta$ as a $\ca$-prebalancing, to emphasize the fact that $\beta$ is related with the action of $\ca$.

\begin{defi} The \textit{relative end} of the pair $(S,\beta)$ is an object $E\in \Ac$ equipped with dinatural transformations $\pi_M: E\xrightarrow{ . .} S(M,M)$ such that 
\begin{equation}\label{dinat-module} S(\ev_X\triangleright \id_M,  \id_M) \pi_M= S(m_{X^*,X,M},  \id_M) \beta^X_{X\triangleright M,M} \pi_{X\triangleright M},
\end{equation}
for any $X\in \ca, M\in \Mo$, and it is universal among those objects with this property. That is, if $E'\in \Ac$ is another object, with dinatural transformations $\pi'_M:\widetilde{E}\xrightarrow{ . .} S(M,M)$, such that they verify \eqref{dinat-module}, then there exists a unique morphism $h:\widetilde{E}\to E$ such that $\pi'_M= \pi_M\circ h $. The \textit{relative coend} is defined similarly. We refer to \cite{BM} for its definition. 
\end{defi}
We will denote the relative end  as $\oint_{M\in \Mo} (S,\beta)$, or simply as $\oint_{M\in \Mo} S$, when the prebalancing $\beta$ is understood from the context.
In the next result we collect some facts about the relative end that will be needed later. The reader is referred to \cite[Prop. 3.3]{BM}, \cite[Prop. 4.2]{BM}.

\begin{teo}\label{compilat-relat-end} Assume  $\Mo, \No$ are  left $\ca$-module categories, and $S, \widetilde{S}:\Mo^{\op}\times \Mo\to \Ac$ are functors equipped with $\ca$-prebalancings $$\beta^X_{M,N}: S(M,X\triangleright N)\to S(X^*\triangleright M,N),$$ $$ \widetilde{\beta}^X_{M,N}: \widetilde{S}(M,X\triangleright N)\to \widetilde{S}(X^*\triangleright M,N),$$ $X\in \ca, M,N\in \Mo$. The following assertions holds
\begin{itemize}

\item[(i)] Assume that the module ends $\oint_{M\in \Mo} (S,\beta), \oint_{M\in \Mo} (\widetilde{S},\widetilde{\beta})$ exist and have dinatural transformations $\pi, \widetilde{\pi}$, respectively. If  $\gamma:S\to \widetilde{S}$ is a natural transformation  such that 
\begin{equation}\label{gamma1}
 \widetilde{\beta}^X_{M,N} \gamma_{(M,X\triangleright N)}=\gamma_{(X^*\triangleright M,N)} \beta^X_{M,N},
\end{equation}
then there exists a unique map $\widehat{\gamma}: \oint_{M\in \Mo} (S,\beta)\to \oint_{M\in \Mo} (\widetilde{S},\widetilde{\beta})$ such that   $\widetilde{\pi}_M \widehat{\gamma}= \gamma_{(M,M)} \pi_M$ for any $M\in \Mo$. If $\gamma$ is a natural isomorphism, then  $\widehat{\gamma}$ is an isomorphism.
\item[(ii)] If the end  $\oint_{M\in \Mo} (S,\beta)$ exists, then for any object $U\in \Ac$,  the end $\oint_{M\in \Mo} \Hom_\Ac(U, S( -, -))$ exists, and there is an isomorphism
$$\oint_{M\in \Mo} \Hom_\Ac(U, S( -, -))\simeq  \Hom_\Ac(U, \oint_{M\in \Mo} (S,\beta) ).$$
Moreover, if $\oint_{M\in \Mo} \Hom_\Ac(U, S( -, -))$ exists for any $U\in \Ac$, then the end $\oint_{M\in \Mo} (S,\beta)$ exists.

\item[(iii)] For any pair of $\ca$-module functors $(F, c), (G, d):\Mo\to \No,$ the functor $$\Hom_{\No}(F(-), G(-)): \Mo^{\op}\times \Mo\to \vect_\ku$$
has a canonical $\ca$-pre-balancing  given by
\begin{equation}\label{beta-for-homs}
 \beta^X_{M,N}: \Hom_{\No}(F(M), G(X\triangleright N))\to \Hom_{\No}(F(X^*\triangleright M), G(N))
\end{equation}
$$ \beta^X_{M,N}(\alpha)= (ev_X\triangleright \id_{G(N)}) m^{-1}_{X^*,X,G(N)} (\id_{X^*}\triangleright d_{X,N}\alpha)c_{X^*,M},$$
for any $X\in \ca, M, N\in \Mo$.
There is an isomorphism $$ \Nat_{\!m}(F,G)\simeq \oint_{M\in \Mo} (\Hom_{\No}(F(-), G(-)), \beta).$$\qed
\end{itemize}
\end{teo}

\begin{rmk}\label{prebalancing-for-hom(u)} If the functor $S$ has $\ca$-prebalancing given by $\beta$, then, according to \cite{BM},  
 for any $U\in \Ac$, the prebalancing for the functor $\Hom_\Ac(U, S( -,-)) $ is defined as
$$\beta^U_{X,M,N}: \Hom_\Ac(U, S( M, X\triangleright N))\to \Hom_\Ac(U, S( X^*\triangleright M, N)),$$
$$\beta^U_{X,M,N}(f)=\beta^X_{M,N}\circ f. $$
\end{rmk} 
\begin{rmk}\label{prebalancing-for-natm} According to \cite[Prop. 4.2]{BM}, the dinatural transformations associated with the end given in Theorem \ref{compilat-relat-end} (iii) are $$\pi_M: \Nat_{\!m}(F,G)\to \Hom_{\No}(F(M), G(M)),\, \pi_M(\alpha)=\alpha_M.$$
\end{rmk} 

Let $\Do \subseteq \ca$ be a tensor subcategory of $\ca$. If $\Mo$ is a left $\ca$-module category, we can consider the restricted $\Do$-module category $\Res^\ca_\Do \Mo$.  We can consider both relative ends 
$$\pi_M:  \oint_{M\in \Mo} (S,\beta)\to S(M,M),$$
$$\lambda_M:  \oint_{M\in  \Res^\ca_\Do\Mo} (S,\beta)\to S(M,M).$$
The next result is a straightforward consequence of the definition of the relative ends.
 \begin{prop}\cite[Prop. 3.12]{BM}\label{restriction-sub}
  Let $S:\Mo^{\op}\times \Mo\to \Ac$ be a functor equipped with a prebalancing $\beta^X_{M,N}: S(M,X\triangleright N)\to S(X^*\triangleright M,N)$. 
 There exists a monomorphism in $\Ac$
$$\iota: \oint_{M\in \Mo} (S,\beta) \hookrightarrow\oint_{M\in \Res^\ca_\Do \Mo} (S,\beta),$$
such that $ \lambda_M \circ \iota= \pi_M$.
 \qed
 \end{prop}

The next result says that, the relative end coincides with the usual end in the case $\ca=\vect_\ku$.  Recall from Section \ref{Section:vect-action} that, any category has a canonical module structure over $\vect_\ku$ that we denote by $\blacktriangleright$.

\begin{prop}\label{add-g}\cite[Prop. 3.14]{BM}\label{relative=usual} Let $\Mo, \Bc$ be categories together with a functor $S:\Mo^{\op}\times \Mo\to \Bc$. The assertions following holds.
\begin{itemize}
\item[(i)] The functor $S$ posses a canonical $\vect_\ku$-prebalancing\begin{equation*} \gamma^V_{M,N}: S(M,X\blacktriangleright N)\to S(X^*\blacktriangleright M,N),
\end{equation*}

\item[(ii)] Any dinatural transformation $\pi_M:E\to S(M,M)$ satisfies equation \eqref{dinat-module} for the canonical action of $\vect_\ku$ on $\Mo$. That is
\begin{equation} S(\ev_V\blacktriangleright \id_M,  \id_M) \pi_M= S(m_{V^*,V,M},  \id_M) \gamma^V_{V\blacktriangleright M,M} \pi_{V\blacktriangleright M},
\end{equation}
for any $V\in\vect_\ku $, $M\in\Mo$.
\end{itemize}\qed
\end{prop}

\section{Relative adjoint algebras}

For the rest of this Section $\ca$ will be a  tensor category and $\Mo$ be an indecomposable exact left $\ca$-module category, with structure functor given by $\rho:\ca\to \End(\Mo)$. 

For each tensor subcategory $\Do \subseteq \ca^*_\Mo$ we shall construct a connected \'etale algebra $\Ac_{\Do, \Mo}$ in the center $Z(\ca)$. When $\Do=\vect_\ku$ is the tensor subcategory of $\ca^*_\Mo$ generated by the unit object, the algebra $\Ac_{\vect_\ku, \Mo}$ coincides with Shimizu's adjoint algebra $\Ac_\Mo$.

\subsection{Computing  adjoint of the action functor}

Take $\Do \subseteq \ca^*_\Mo$ a tensor subcategory.

We shall denote by $\End(\Mo)$ the category of right exact functors $F:\Mo\to \Mo$. Then, the category $\End(\Mo)$ has an obvious structure of $\Do$-bimodule category. Hence, we can consider the relative center $Z_\Do(\End(\Mo))$. Let us explicitly describe objects in this category. An object in $Z_\Do(\End(\Mo))$ is a pair $(F, \gamma)$, where $F:\Mo\to \Mo$ is a right exact functor, $\gamma_{(D,d)}: F\circ D\to D\circ F$ is a family of natural isomorphisms, for any $(D,d)\in \Do$, such that
\begin{equation}\label{relativ-funct} (\id_{(D,d)} \circ \gamma_{(D',d')})(\gamma_{(D,d)}\circ \id_{(D',d')})=\gamma_{(D,d)\circ (D',d')},
\end{equation}
for any $(D,d), (D',d')\in \Do$.   Since $\Mo$ is a left module category over $\ca^*_\Mo$, in particular, it is a $\Do$-module category. Turns out that $Z_\Do(\End(\Mo))=\Do^*_\Mo. $

\begin{lema}\label{c-bimodule-relative}  The following assertions hold.
\begin{itemize}
\item[1.] The category $Z_\Do(\End(\Mo))$ has structure of $\ca$-bimodule category. 

\item[2.] For any $X\in \ca$, the functor $\rho(X)$ belongs to $Z_\Do(\End(\Mo))$. Thus we can consider the functor $\rho:\ca\to Z_\Do(\End(\Mo))$.   Turns out that this functor is monoidal, exact and faithful. 
\end{itemize}

\end{lema}
\pf  1. Take $X\in \ca$, $(F, \gamma)\in Z_\Do(\End(\Mo))$, and define
$$X\triangleright (F, \gamma)=(X\triangleright F, {}^X\gamma), \, (F, \gamma)\triangleleft X=(F\triangleleft X, \gamma^X),$$ where
$$X\triangleright F,\; F\triangleleft X :\Mo\to \Mo, $$
$$(X\triangleright F)(M)=X \triangleright F(M), \,  (F\triangleleft X)(M)=F(X\triangleright M), $$
for any $M\in \Mo$. For any $(D,d)\in \Do$ define the natural isomorphism
$${}^X\gamma_{(D,d)}: (X\triangleright F)\circ D\to D\circ (X\triangleright F),$$
$$\big({}^X\gamma_{(D,d)}\big)_M=d^{-1}_{X,F(M)} (\id_X \triangleright (\gamma_D)_M), $$
for any $M\in \Mo$. Also define, for any $(D,d)\in \Do$
$$\gamma^X_{(D,d)}: (X\triangleright F)\circ D\to D\circ  (X\triangleright F), $$
$$\big(\gamma^X_{(D,d)} \big)_M= (\gamma_{(D,d)})_{X\triangleright M} F(d^{-1}_{X,M}),  $$
for any $M\in \Mo$. 

2. It follows from item (1), since $\rho(X)=X\triangleright \Id$. Exactness of the functor $\rho$ follows from the exactness of the bifunctor $\triangleright$ in the first variable. There is a commutative diagram of functors
\begin{equation}\label{diagram11}
\xymatrix{&\ca 
\ar[dl]_{\rho}
\ar[dr]^{X\mapsto X \triangleright -}&\\
 Z_\Do(\End(\Mo)) \ar[rr]^{\uc}&&  \End(\Mo),}
\end{equation}
where $\uc: Z_\Do(\End(\Mo)) \to \End(\Mo)$ is the forgetful functor. Since the action functor $\ca\to \End(\Mo)$ is faithful, it follows that $\rho$ is faithful.

\epf

\begin{rmk}\label{left-action-functor} For future reference, for any $X\in \ca$, we denote the functor $L^X=\rho(X)\in \End(\Mo).$ For any $M\in \Mo$, $L^X(M)=X\triangleright M. $  It follows from part (2) of the above Lemma that $(L^X, l^X)\in  Z_\Do(\End(\Mo))$, where
$$l^X_D: L^X\circ D\to D\circ   L^X,$$
$$\big(l^X_D\big)_M=d^{-1}_{X,M}, $$
for any $(D,d)\in \Do$.
\end{rmk}

The adjoint algebra associated to $\Mo$, introduced in  \cite{Sh2},  was defined using the right adjoint of the action functor $\rho:\ca\to \End(\Mo)$. Turns out that the right adjoint is given by
$$\End(\Mo)\to \ca, \,  F\mapsto \int_{M\in \Mo} \uhom(M, F(M)).$$
In \textit{loc. cit.} the author defines
$$\A_\Mo= \int_{M\in \Mo} \uhom(M, M)$$

 We intend to generalize this construction using the functor $\rho:\ca\to Z_\Do(\End(\Mo))$ instead, and replace the usual end with the \textit{relative} end. Let us show first that the bifunctor $\uhom(-, F(-))$ posses a prebalancing.

\begin{lema}\label{inthomfunct-preb} For any $(F,\gamma)\in Z_\Do(\End(\Mo))$ let us denote by $S^F$ the functor given by
$$S^F: \Mo^{\op}\times \Mo\to \ca, (M,N)\mapsto \uhom(M, F(N)). $$
It has a canonical $\Do$-prebalancing $\beta^D_{M,N}$ as follows. For any $D\in \Do$, $\beta$ is the composition
$$S^F(M, D(N))=\uhom(M, F(D(N))) \xrightarrow{\uhom(M, (\gamma_D)_N)}   \uhom(M, D(F(N))) \xrightarrow{} $$
$$\xrightarrow{\xi^D_{M,F(N)}} \uhom(D^{\la}(M), F(N))= S^F(D^{\la}(M), N)$$
Here, isomorphism $\xi^D$ is the one described in Lemma \ref{int-hom-adjoint}. \qed
\end{lema}

The next result is a generalization of a result of K. Shimizu. When applied to  $\Do=\vect_\ku$, it recovers  exactly  \cite[Thm. 3.4]{Sh2}. However, the proof of this result is more involved  as the one presented in {\it loc.cit.}, mainly due to the presence of  prebalancings.

\begin{teo}\label{adjoint-structure}  Let $\Mo$ be an exact indecomposable $\ca$-module category. The following assertions hold.

\begin{itemize}
\item[1.] The functor $\bar{\rho}: Z_\Do(\End(\Mo))\to \ca$ 
\begin{equation}\label{radj-rho}\bar{\rho}(F)= \oint_{M\in \Mo} \uhom(M, F(M))=\oint_{M\in \Mo} (S^F, \beta),
\end{equation}  
is well-defined, and it is a right adjoint of the functor $\rho:\ca\to Z_\Do(\End(\Mo))$. 
\item[2.]   The functor $\bar{\rho}$ is exact and faithful.

\item[3.]   The functor $\rho:\ca\to Z_\Do(\End(\Mo))$ is dominant. 
\end{itemize}
\end{teo}
\pf 1. Let us prove first that the right adjoint to $\rho$ is given by formula \eqref{radj-rho} on objects. Let be  $X\in \ca$, $(F,\gamma)\in  Z_\Do(\End(\Mo))$. We have the functor 
$$ \Hom_{\Mo}(X\triangleright -,F(-)) :\Mo^{\op}\times \Mo \to vect_\ku,$$
with $\Do$-prebalancing given by
\begin{align}\begin{split}\label{prebalancing-hom1} &\beta^X_{D,M,N}: \Hom_{\Mo}(X\triangleright M,F(D(N))) \to \Hom_{\Mo}(X\triangleright D^{\la}(M),F(N)),\\
&\beta^X_{D,M,N}(f)=(ev_D)_{F(N)} D^{\la}( (\gamma_D)_N f) \tilde{d}^{-1}_{X,M},
\end{split}
\end{align}
for any $(D,d)\in \Do$. Here $ev_D:D^{\la}\circ D \to \Id$ is the evaluation of the adjunction $(D^{\la}, D)$, and $\tilde{d}_{X,M}: D^{\la}(X\triangleright M)\to X\triangleright D^{\la}( M)$ is the $\ca$-module structure of the functor $D^{\la}$. The formula for this prebalancing is taken from \eqref{beta-for-homs}. We have natural isomorphisms
\begin{align}\label{chain-is}
\Hom_\ca(X,\rho^{\ra}(F))&\simeq \Nat_m(\rho(X),F)\simeq \oint_{M\in \Mo} \Hom_\Mo(X \triangleright M, F(M))
\end{align}
The first isomorphism is just the definition of the right adjoint, and the second isomorphism follows from Theorem \ref{compilat-relat-end} (iii).

For any $X\in \ca$, we also have the functor $$  \Hom_{\ca}(X,\uhom(-,F(-))) :\Mo^{\op}\times \Mo \to vect_\ku.$$
This functor has a $\Do$-prebalancing given
$$b^X_{D,M,N}: \Hom_{\ca}(X,\uhom(M,F(D(N)))) \to \Hom_{\ca}(X,\uhom(D^{\la}(M),F(N))), $$
$$ b^X_{D,M,N} (f)=\beta^D_{M,N} f. $$
Here $\beta^D_{M,N} $ is the prebalancing of the functor $S^F$ given in Lemma \ref{inthomfunct-preb}.
Recall the natural isomorphisms presented in \eqref{Hom-interno}
$$\phi^X_{M,N}:\Hom_{\ca}(X,\uhom(M,N))\to \Hom_{\Mo}(X\triangleright M,N).$$
\begin{claim}\label{claim-thm1} For any $X\in \ca$ we have $$b^X_{D,M,N}  \phi^X_{M,F(D(N))}= \phi^X_{D^{\la}(M),F(N)} \beta^X_{D,M,N}.\qed$$
\end{claim}
The proof of this claim is given in the Apendix in Proposition \ref{proofclaim1}. Using this claim, it follows from Theorem \ref{compilat-relat-end} (i) that there exist an isomorphism 
\begin{align*}
 \widehat{\phi}^X: \oint_{M\in \Mo} \Hom_\ca(X, \uhom(M,F(M)))
\to  \oint_{M\in \Mo} \Hom_\Mo(X \triangleright M, F(M))
\end{align*}
such that the diagram
\begin{equation}\label{about-phi.preb}
\xymatrix{
\oint_{M\in \Mo} \Hom_\ca(X, \uhom(M,F(M)))\ar[d]^{\pi^{X,F}_M}\ar[rr]^{\widehat{\phi}^X}&&  \oint_{M\in \Mo} \Hom_\Mo(X \triangleright M, F(M))\ar[d]_{\widetilde{\pi}^{X,F}_M} \\
\Hom_\ca(X, \uhom(M,F(M))) \ar[rr]^{\phi^X_{M,N}}&& \Hom_\Mo(X \triangleright M, F(M))}
\end{equation}
is commutative. Here $ \pi^{X,F}_M$ and $\widetilde{\pi}^{X,F}_M $ are the corresponding dinatural transformations. 
 Hence, using isomorphisms \eqref{chain-is}
\begin{align*}\Hom_\ca(X,\rho^{\ra}(F))&\simeq\oint_{M\in \Mo} \Hom_\ca(X, \uhom(M,F(M))) \\
&\simeq \Hom_\ca(X, \bar{\rho}(F)).
\end{align*}
The second isomorphism follows from Theorem \ref{compilat-relat-end} (ii). This proves that the corresponding relative end exists and it is a right adjoint of the functor $\rho$.

Let us see how to define $\bar{\rho}$ on morphisms. Note that, under the identification 
$$\oint_{M\in \Mo} \Hom_\ca(X, \uhom(M,F(M))) \simeq \Hom_\ca(X, \bar{\rho}(F)),$$
if $\pi^F_M:  \oint_{M\in \Mo} \uhom(M, F(M))\to \uhom(M, F(M)) $ is the associated dinatural transformation, then
\begin{equation}\label{relation-pis}
\pi^{X,F}_M(g)=\pi^F_M \circ g,
\end{equation}
for any $g\in \Hom_\ca(X, \bar{\rho}(F)).$ This follows from the proof of Theorem \ref{compilat-relat-end} (ii).
Let $f:(F,\gamma)\to (G, \delta)$ be a morphism in  $Z_\Do(\End(\Mo))$. This  map induces a natural morphism $\widetilde{f}: S^F\to S^G$,  $\widetilde{f}_{M,N}=\uhom(\id_M, f_N)$, for any $M, N\in \Mo$. Using the naturality of $\xi$ (see Lemma \ref{int-hom-adjoint}), one can prove that
$$\beta^D_{M,N}  \widetilde{f}_{M,D(N)}= \widetilde{f}_{D^{\la}(M),N}\beta^D_{M,N},$$
for any $D\in \Do$, $M, N\in \Mo$. This implies, using Theorem \ref{compilat-relat-end} (i) that, there exists a morphism $\bar{\rho}(f):\oint_{M\in \Mo} (S^F, \beta)\to \oint_{M\in \Mo} (S^G, \beta)$ such that
\begin{align}\label{morph-rho}\widetilde{f}_{M,M} \pi^F_M= \pi^G_M \bar{\rho}(f),
\end{align}
for any $M\in \Mo.$ This finishes the proof that the functor $\bar{\rho}$ is well-defined on morphisms. 

\medbreak

2. Next, we shall prove that $\bar{\rho}$ is exact. For this, we shall prove that $\bar{\rho}$ preserves kernels and cokernels. Let $(F,\gamma), (G, \delta)$ be objects in  $Z_\Do(\End(\Mo))$, and $f:(F,\gamma)\to (G, \delta)$ be a morphism. Let $k:(K,\lambda)\to (F,\gamma)$ be a kernel of $f$. So, we have an exact sequence
$$0 \to (K,\lambda) \xrightarrow{k} (F,\gamma) \xrightarrow{f} (G, \delta).$$
We shall prove that $\ker(\bar{\rho}(f))= \bar{\rho}(k).$
Let $h: Q\to \bar{\rho}(F)$ be a morphism in $\ca$ such that $$\bar{\rho}(f)\circ h=0.$$ For any $M\in \Mo$, the map $\pi^F_M\circ h\in \Hom_\ca(Q, \uhom(M,F(M)))$. Using the natural isomorphism \eqref{Hom-interno}, set
$$ \alpha_M=\phi^Q_{M,F(M)}(\pi^F_M h):Q\triangleright M\to F(M). $$
For any $X\in \ca$,  define $L^X:\Mo\to \Mo$, $L^X(M)=X\triangleright M$. Using Lemma \ref{c-bimodule-relative} (2), one can see that $L^X\in Z_\Do(\End(\Mo))$. See also Remark \ref{left-action-functor}.
\begin{claim} $\alpha: L^Q\to F$ is a morphism in $ Z_\Do(\End(\Mo))$ such that $f\circ \alpha=0$.
\end{claim}
\begin{proof}[Proof of claim] It follows from commutativity of diagram \eqref{about-phi.preb} and equation \eqref{relation-pis} that
$$\alpha_M=\phi^Q_{M,F(M)}(\pi^{Q,F}_M(h))= \widetilde{\pi}^{Q,F}_M( \widehat{\phi}^Q(h)).$$
Using identification 
$$ \Nat_{\!m}(Q\triangleright -,F)\simeq  \oint_{M\in \Mo} \Hom_\Mo(Q \triangleright M, F(M)),$$
and Remark \ref{prebalancing-for-natm}, it follows that $\alpha$ indeed is a natural module transformation.
For any $M\in \Mo$
\begin{align*} f_M  \alpha_M &=f_M \phi^Q_{M,F(M)}(\pi^F_M h)= \phi^Q_{M,G(M)}(\uhom(\id_M, f_M)\pi^F_M h)\\
&=\phi^Q_{M,G(M)}(\pi^G_M  \bar{\rho}(f) h)=0.
\end{align*}
The second equality is commutativity of diagram \eqref{comm-phi2}, the third equality follows from \eqref{morph-rho}. This finishes the proof of the Claim.
\end{proof}
Since $k:(K,\lambda)\to (F,\gamma)$ is the kernel of $f$, there exists a morphism $\mu: L^Q\to K$ in $ Z_\Do(\End(\Mo))$ such that
\begin{equation}\label{ab-k} k\circ \mu= \alpha. 
\end{equation}
The morphism
$$ \psi^Q_{M,K(M)}(\mu_M): Q\to \uhom(M, K(M)),$$
is a dinatural transformation. This fact can be checked using diagrams \eqref{comm-psi1}, \eqref{comm-psi2}.
\begin{claim}\label{claim-prebal} The dinatural transformation $ \psi^Q_{M,K(M)}(\mu_M)$  satisfies \eqref{dinat-module} with respect to the prebalancing of the functor $S^K$. That is, the following equation is fulfilled 
\begin{align*} \uhom((\ev_D)_M, &\id_{K(M)} ) \psi^Q_{M,K(M)}(\mu_M)=\\
&=\xi^D_{D(M),K(M)}\uhom(\id_{D(M)}, (\lambda_D)_M) \psi^Q_{D(M),K(D(M))}(\mu_{D(M)}),
\end{align*}
for any $D\in \Do$, $M\in \Mo$. Here $\ev_D: D^{\la} \circ D\to \Id$ is the evaluation of the adjunction $(D^{\la}, D)$.\qed
\end{claim}
The proof of this claim is given in the Apendix, in Proposition \ref{proofclaim2}. It follows from Claim \ref{claim-prebal}, and the universal property of the relative end, that, there exists  a morphism $\widehat{h}:Q\to \bar{\rho}(K)$ such that
\begin{equation}\label{defi-hat-h}
\pi^K_M\circ \widehat{h} = \psi^Q_{M,K(M)}(\mu_M),
\end{equation} 
for any $M\in \Mo$. We shall prove that 
\begin{equation}\label{hh11} \bar{\rho}(k)  \widehat{h} =h.
\end{equation}
This will end the proof that  $\ker(\bar{\rho}(f))= \bar{\rho}(k).$ Indeed, applying $\pi^F_M$ to the left hand side of \eqref{hh11} we get
\begin{align*}  \pi^F_M \bar{\rho}(k)  \widehat{h} &=\uhom(\id_M, k_M) \pi^K_M \widehat{h} \\
&= \uhom(\id_M, k_M) \psi^Q_{M,K(M)}(\mu_M)\\
&= \psi^Q_{M,F(M)}(k_M\mu_M)= \psi^Q_{M,F(M)}(\alpha_M)\\
&=  \pi^F_M h
\end{align*}
The first equality follows from \eqref{morph-rho}, the second equality follows from \eqref{defi-hat-h}, the third one follows from \eqref{comm-psi1}, the fourth equality is  \eqref{ab-k} and the last one is the definition of $\alpha$.  Equation \eqref{hh11} then follows from the universal property of the dinatural transformation $\pi^F. $

This finishes the proof that $\ker(\bar{\rho}(f))= \bar{\rho}(k).$ The proof that $\bar{\rho}$ preserves cokernels is done similarly. Hence $\bar{\rho}$ is exact. 

\bigbreak

Let us prove that $\bar{\rho}$ is faithfull. Since it is exact, it is enough to prove  that reflects zero objects. Take a non-zero object $(F, \gamma)\in  Z_\Do(\End((\Mo))$. Again, lets keep in mind the  commutative diagram 
\begin{equation}\label{diagram1}
\xymatrix{&\ca 
\ar[dl]_{\rho}
\ar[dr]^{X\mapsto X \triangleright -}&\\
 Z_\Do(\End((\Mo)) \ar[rr]^{\uc}&&  \End((\Mo),}
\end{equation}
where $\uc: Z_\Do(\Rex(\Mo)) \to \End((\Mo)$ is the forgetful functor.  Since $\Mo$ is an exact indecomposable as a $\ca$-module category, it follows from \cite[Prop. 2.6 (ii)]{EG} that the functor $\ca\to  \Rex(\Mo)$, $X\mapsto X \triangleright -$ is dominant. Whence, there exists an object $Y\in \ca$ such that $$\Nat(Y\triangleright -,  \uc(F, \gamma))\neq 0.$$ Thus
$$0\neq \Nat(Y\triangleright -,  \uc(F, \gamma))\simeq  \Hom_{Z_\Do(\End((\Mo))}( I(Y\triangleright -) , (F, \gamma)). $$
Here $I: \Rex(\Mo)\to Z_\Do(\Rex(\Mo))$ is the left adjoint of the forgetful functor, and it is given by
$$I(F)=\int^{X\in \ca} X\triangleright F \triangleleft {}^*X. $$
The functor $I$ was presented in Proposition \ref{adjoint-forgetf}. Let us denote 
$$ Z=\int^{X\in \ca} X\ot Y \ot {}^*X. $$
Then, $I(Y\triangleright -)\simeq\rho(Z)$. 
Thus
$$0\neq   \Hom_{Z_\Do(\Rex(\Mo))}( \rho(Z), (F, \gamma)) \simeq \Hom_\ca( Z, \bar{\rho} (F, \gamma)),$$
which implies that $\bar{\rho} (F, \gamma) \neq 0$.

 \medbreak
3. It follows from Proposition \ref{dominant-monads} since $\rho$ is faithful.
\epf

\subsection{\'Etale algebras corresponding to tensor subcategories of $\ca^*_\Mo$ }

In this Section we shall construct étale algebras out of tensor subcategories. This construction was already  presented in  \cite[Thm. 4.10]{DMNO}. Let $\Do \subseteq \ca^*_\Mo$ be a tensor subcategory.  Observe that, since $Z_\Do(\End(\Mo))$ is a $\ca$-bimodule category (see Lemma \ref{c-bimodule-relative}), then we can consider the relative center $Z_\ca(Z_\Do(\End(\Mo)))=Z_\Do(\ca^*_\Mo).$

\begin{prop}\label{z-rho} The following assertions hold.
\begin{itemize}
\item[1.] The functor $\rho:\ca\to Z_\Do(\End(\Mo))$ is a $\ca$-bimodule functor. 

\item[2.]  Using (1) we can consider the functor $Z(\rho):Z(\ca) \to Z_\Do(\ca^*_\Mo).$ This tensor functor is dominant.

\end{itemize}
\end{prop}
\pf The proof of (1) follows by a straightforward computation. It follows from Proposition \ref{about-Z} that $Z(\bar{\rho})$ is a right adjoint to $Z(\rho)$, and this functor is exact and faithful. It follows from Proposition \ref{dominant-monads} that $Z(\rho)$ is dominant.
\epf
Now we can give the main definition of the paper. 
\begin{defi}\label{def-a1}   Let $\Do \subseteq \ca^*_\Mo$ be a tensor subcategory. The object 
$$\Ac_{\Do, \Mo}=Z(\bar{\rho})(\Id)=\oint_{M\in \Mo} (\uhom(-,-), \beta)\in Z(\ca)$$
is called the \textit{adjoint algebra of $\Mo$ relative} to $\Do$. The morphisms $\pi_M:\Ac_{\Do, \Mo}\to  \uhom(M,M)$ are the dinatural transformations associated to this relative end. 
\end{defi}

\begin{rmk} Since the relative end coincides with the usual end when $\Do=\vect_\ku$, then $\Ac_{\vect_\ku, \Mo}=\Ac_{ \Mo}$ is the adjoint algebra associated to $\Mo$, as introduced by Shimizu \cite{Sh2}.
\end{rmk} 

The proof of the next result follows \textit{mutatis mutandis} from the proof presented in \cite{Sh2}.
\begin{prop}\label{alg-braid-adj} The object $\Ac_{\Do, \Mo}$ has a half-braiding, multiplication  and unit
$$\sigma_X: \Ac_{\Do, \Mo}\ot X\to X\ot \Ac_{\Do, \Mo}, $$
$$m: \Ac_{\Do, \Mo}\ot \Ac_{\Do, \Mo}\to \Ac_{\Do, \Mo},\quad u:\uno\to \Ac_{\Do, \Mo},$$
 determined by 
 \begin{equation}\label{braid-din}
 \xymatrix@C=90pt@R=16pt{
      \cha_\Mo \otimes X
      \ar[r]^-{\pi^\Mo_{X \triangleright M} \otimes \id_X}
      \ar[dd]_{\sigma^{\Mo}_X}
      & \uhom(X \triangleright M,  X \triangleright  M) \otimes X
      \ar[d]^{\mathfrak{b}_{X, M, X \triangleright  M}} \\
      & \uhom(M,  X \triangleright  M)
      \ar[d]^{ \mathfrak{a}_{X,M,M}} \\
      X \otimes \cha_\Mo
      \ar[r]^-{\id_X \otimes \pi^\Mo_M}
      & X \ot \uhom(M, M)
    }
 \end{equation}
 \begin{equation}\label{mult-din1}\pi_M\circ m=\comp^\Mo_{M,M} (\pi_M\ot \pi_M), \,\, \pi_M\circ u= \coev^\Mo_{ M},
 \end{equation}  
 for any $X\in \ca$, $M\in \Mo$. Moreover, $\Ac_{\Do, \Mo}$ is a commutative algebra, that is $ m\circ \sigma_{\Ac_{\Do, \Mo}}=m.$ \qed
\end{prop}

\begin{rmk} Equations \eqref{braid-din}, \eqref{mult-din1} are the same as equations \eqref{half-braidig-ch} and \eqref{product-cha-M}, describing the product and the half-braiding in terms of the dinatural transformations.
\end{rmk}

\begin{prop}\label{et-1} Le $\Mo$ be an exact indecomposable module category over $\ca$, and let  $\Do \subseteq \ca^*_\Mo$ be a tensor subcategory. The following statements hold:
\begin{itemize}

\item[1.]  The object  $\Ac_{\Do, \Mo}$ is a connected \'etale algebra  in the center $\Zc(\ca)$ such that there is a monoidal equivalence
$$Z(\ca)_{A_{\Do,\Mo}} \simeq Z_\Do(\ca^*_\Mo).$$
\item[2.]
 If $\Do_1\subseteq  \Do_2 \subseteq \ca^*_\Mo$ are tensor subcategories, then there is an algebra inclusion
$$ \Ac_{\Do_2, \Mo} \subseteq \Ac_{\Do_1, \Mo}.$$

\item[3.] $$\fpd_{Z(\ca)}(\Ac_{\Do, \Mo})= \frac{\fpd(\ca)}{\fpd(\Do)}.$$
\end{itemize} 
\end{prop}
\pf 1. Consider the adjoint functors $(Z(\rho), Z(\bar{\rho}))$. It follows from Proposition \ref{z-rho} that they satisfy conditions of Corollary \ref{from adj-to-etale}. Whence, the algebra $Z(\bar{\rho})(\Id)=\Ac_{\Do, \Mo}$ is étale.
It follows from Proposition \ref{central-adj-alg} that there is a tensor equivalence $Z(\ca)_{\Ac_{\Do,\Mo}} \simeq Z_\Do(\ca^*_\Mo).$

2. It follows from Proposition \ref{restriction-sub}.

3. Since $\fpd(Z_\Do(\ca^*_\Mo))=\fpd(\Do)\fpd(\ca), $ the formula for the dimension of $\Ac_{\Do, \Mo}$ follows from \cite[Lemma 6.2.4]{EGNO}.

\epf

\section{Some concrete computations}\label{Section:comp}
For the rest of this Section  $\ca$ will denote  a tensor category. We shall give some explicit computations of the relative adjoint $\Ac_{\Do, \Mo}$ for diverse $\ca$-module categories $\Mo$ and tensor subcategories $\Do\hookrightarrow  \ca^*_\Mo$.

\medbreak

To compute examples of such objects, let us describe, more explicitly, what equation \eqref{dinat-module} looks like in this particular case. 

In this case, the relative end is applied to the functor
$$\uhom(-,-): \Mo^{\op}\times \Mo\to \ca. $$  
If $\Do\hookrightarrow  \ca^*_\Mo$ is a tensor subcategory, then $\Do$ acts on $\Mo$ by evaluation. The functor $\uhom(-,-)$ has a $\Do$-prebalancing 
$$\xi^D_{M,N}: \uhom(M, D(N))\to \uhom(D^{\la}(M), N),$$
described in Lemma \ref{int-hom-adjoint}. If we denote
$$\ev^D:D^{\la}\circ D\to \Id,$$
the counit if the adjunction $(D^{\la}, D)$, then, in this case, a dinatural transformation $\pi_M:E\to  \uhom(M,M)$ satisfies equation \eqref{dinat-module}  if and only if
\begin{equation}\label{dinat-module2} \uhom((\ev^D)_M,  \id_M) \pi_M= \xi^D_{D( M),M} \pi_{D( M)},
\end{equation}
for any $M\in \Mo$.

\subsection{Case $\Rep(G/N)\hookrightarrow (\ca\rtimes G)^*_{\ca}$}

Let $G$ be a finite group acting on $\ca$ by tensor autoequivalences $(F_g,\xi^g):\ca\to \ca$. Here, for any $g\in G$
$$\xi^g_{X,Y}: F_g(X)\ot F_g(Y)\to F_g(X\ot Y) $$
are the natural isomorphisms giving the monoidal functor structure on $F_g$.
We can consider  the
$G$-\textit{crossed product tensor category} $\ca\rtimes G$. 

\medbreak

As an abelian category $\ca\rtimes G=
\bigoplus_{g\in G}\ca_g$, where $\ca_g =\ca$. An object in $\ca_g$ will be denoted by $[X, g]$, where $X\in \ca$. Similarly, a morphism in  $\ca\rtimes G$ will be denoted by $$[f,g]: [X,g]\to [Y,g]$$
where $f:X\to Y$ is a morphism in $\ca$.
The tensor product of $\ca\rtimes G$ is determined by
$$[X, g]\otimes [Y,h]:= [X\otimes F_g(Y),
gh],\  \  \   X,Y\in \ca,\  \   g,h\in G,$$ and
the unit object is $[\uno,1]$. See \cite{Ta} for details on the associativity
constraint and a proof of the pentagon identity.

The category $\ca\rtimes G$ acts on $\ca$ as follows. If $[X, g] \in \ca\rtimes G$ and $Y\in \ca$ then
\begin{equation}
 [X, g]\triangleright Y= X\ot F_g(Y).
\end{equation}
It is well-known that the dual category $(\ca\rtimes G)^*_{\ca}$ is tensor equivalent to the $G$-equivariantization $\ca^G$. 
\begin{lema} For any  normal subgroup $N\triangleleft G$, there is an inclusion of tensor categories
$$ \Rep(G/N)\hookrightarrow (\ca\rtimes G)^*_{\ca}.$$
\end{lema} 
\pf Since the category of finite-dimensional vector spaces $\vect_\ku$ is included in $\ca$ via: $V \mapsto  V\blacktriangleright \uno$. Here $\blacktriangleright: \vect_\ku\times \ca\to \ca$ is the canonical action of $\vect_\ku$ on any finite abelian $\ku$-linear category $\ca$. See Section \ref{Section:vect-action}. Hence $\Rep(G)\hookrightarrow \ca^G$. Since the canonical projection $G \twoheadrightarrow G/N$ induces an inclusion $\Rep(G/N)\hookrightarrow \Rep(G)$, and $\ca^G\simeq  (\ca\rtimes G)^*_{\ca}$, the inclusion $ \Rep(G/N)\hookrightarrow (\ca\rtimes G)^*_{\ca}$ follows. 

\medbreak

To be more explicit, let us describe this inclusion. If $V\in \Rep(G/N)$, then $V\in \Rep(G)$ with action given by
$$ g\cdot v= \overline{g}\cdot v.$$
This implies that, for any $g\in G$, there is an isomorphism
$$s_g: V\blacktriangleright \uno \to V\blacktriangleright  F_g(\uno)\simeq  F_g( V\blacktriangleright \uno),$$
given by the action of $g$ on $V$. For any $V\in \Rep(G/N)$ let us denote by $L_V:\ca\to \ca$ the $\ca\rtimes G$-module functor described as follows. As a functor
$$L_V(Y)= Y\ot (V \blacktriangleright \uno), $$
For any $Y\in \ca$. The module structure of the functor $L_V$ is given by the composition
$$L_V([Y,g]\triangleright X)=Y\ot F_g(X)\ot (V\blacktriangleright \uno)\xrightarrow{\id_{Y\ot F_g(X)}\ot s_g} Y\ot F_g(X)\ot F_g( V\blacktriangleright \uno)\to$$
$$ \xrightarrow{ \id_Y\ot \xi^g_{X, V\blacktriangleright \uno}} Y\ot F_g(X\ot V\blacktriangleright \uno)= [Y,g]\triangleright L_V(X),$$
for any $X\in \ca$, $[Y,g]\in \ca\rtimes G$.
\epf
We aim to compute the relative adjoint $\Ac_{\Rep(G/N),\ca}$ associated with the inclusion  $ \Rep(G/N)\hookrightarrow (\ca\rtimes G)^*_{\ca}$, as an object in $\ca\rtimes G.$ For any $g\in G$ we shall denote
$$\Ac_g= \int_{X\in \ca}  X\ot F_g(X)^*\in \ca,$$
with the corresponding dinatural transformations
$$\pi^g_X:\Ac_g\to   X\ot F_g(X)^*.$$ 
If $N\triangleleft G$ is a normal subgroup, we shall denote
$$\Ac(N,\ca)=\oplus_{n\in N} [\Ac_n,n] \in \ca\rtimes G. $$
\begin{prop} There exists an isomorphisms of objects
$$\Ac_{\Rep(G/N),\ca}\simeq \Ac(N,\ca).$$
\end{prop}
\pf  Let us, explicitly describe, all the necessary ingredients to understand equation \eqref{dinat-module} in this particular example.

 First of all, one can easily check that the internal hom of $\ca$ as a $\ca\rtimes G$-module category is
\begin{equation}\label{pr-h} \uhom(X,Y)=\oplus_{g\in G} [Y\ot F_g(X)^*, g],
\end{equation}
for any $X, Y\in \ca$. For any $V\in\Rep(G/N)$ let us denote by 
$$\beta^V_{M,N}: \uhom(M, N\ot V\blacktriangleright \uno) \to \uhom(M\ot V^*\blacktriangleright \uno, N)$$
the $\Rep(G/N)$-prebalancing 
of the functor 
$$\uhom(-,-): \ca^{\op}\times \ca\to  \ca\rtimes G.$$ 
Here, we are using that the left adjoint to $L_V$ is $L_{V^*}$. Observe that, using \eqref{pr-h}, the prebalancing $\beta^V_{M,N}=\oplus_{g\in G}  \beta^{V,g}_{M,N}$, where
$$\beta^{V,g}_{M,N}:  [N\ot V\blacktriangleright \uno\ot F_g(M)^*, g] \to [N\ot F_g(M\ot V^*\blacktriangleright \uno)^*, g].$$

According to Lemma \ref{inthomfunct-preb} and Lemma \ref{int-hom-adjoint}, one can compute $\beta^{V,g}_{M,N}$ as the composition:
$$N\ot V\blacktriangleright \uno\ot F_g(M)^* \xrightarrow{ \gamma^{V,g}_{M,N}} N\ot F_g(M\ot V^*\blacktriangleright \uno)^* \xrightarrow{t_g} N\ot F_g(M\ot V^*\blacktriangleright \uno)^* $$
Here $\gamma^{V,g}_{M,N}$ is the morphism presented in Proposition \ref{add-g} associated with the additivity of the functor $ \uhom(-,-)$, and $t_g=\id_N\ot F_g(\id_M\ot s_g)$, where $s_g$ is the action of $g$ on $V^*$.

Let $E$ be an object in $\ca\rtimes G$ equipped with dinatural transformations 
$$\lambda_X=\oplus_{g\in G}  \lambda^g_X:E\to \oplus_{g\in G} [X\ot F_g(X)^*, g].$$
The dinatural maps $\lambda$  satisfies equation  \eqref{dinat-module2} if and only if $\lambda^g_X=0$ for any  $g\notin N$.
\epf 

\subsection{Case $\Do\hookrightarrow \ca$ is a tensor subcategory} 

If $\Mo=\ca$ with the regular action, then $\ca^*_\Mo=\ca^{\rev}$. Hence  any tensor subcategory $\Do\subseteq \ca$ produces a tensor subcategory $\Do^{\rev}\subseteq \ca^*_\ca$. In this particular case the funtor $S$ is
\begin{equation}\label{S11} S:\ca^{\op}\times \ca\to \ca,
\end{equation}
$$S(X,Y)=Y\ot X^*.$$
The action of $\Do^{\rev}$ on $\ca$ is given by
$$D\triangleright X=X\ot D,  $$
for any $D\in \Do$, $X\in \ca$. The prebalancing of the functor $S$ is given by
\begin{equation}\label{S12} \beta^D_{M,N}:(N\ot D)\ot M^*\to N\ot (M\ot {}^*D)^*,\end{equation}
$$\beta^D_{M,N}=(\id_N\ot \phi^{-1}_{M, {}^*D}) a_{N,D,M^*}. $$
Here $a$ is the associativity constraint, and $\phi_{X,Y}:(X\ot Y)^*\to Y^*\ot X^*$ is the canonical isomorphism. Note that the right dual in $\Do^{\rev}$ is the left dual in $\Do$. A dinatural transformation $\pi$ satisfies equation \eqref{dinat-module} if and only if
$$\pi_{  X\triangleright M} =(\beta^X_{X\triangleright M,M} )^{-1} S(m^{-1}_{X^*,X,M},  \id_M) S(\ev_X\triangleright \id_M,  \id_M) \pi_M$$
for any $X\in \Do$, $M\in \ca$. And, in this particular example it satisfies \eqref{dinat-module}  if and only if
\begin{align}\label{dinat-case0}\begin{split} \pi_{ M\ot X} &=a^{-1}_{M,X,(M\ot X)^*}  (\id_M\ot \phi_{M\ot X, {}^*X}  ( (\id_M\ot ev_X)a_{M,X,{}^*X})^*) \pi_M
\end{split}
\end{align}
\begin{rmk} Equation \eqref{dinat-case0} implies that the dinatural transformation $\pi$ depends only on objects that belong to a ``quotient"  $\ca/\Do$. This observation will be used in subsection \ref{subs:deligne}, in the example when $\ca$ equals the Deligne tensor product of two tensor categories $\Do\boxtimes \Bc$.
\end{rmk}
Note that this is equivalent to
$$a_{M,X,(M\ot X)^*} \pi_{ M\ot X}= (\id_M\ot \phi_{M\ot X, {}^*X}  ( (\id_M\ot ev_X)a_{M,X,{}^*X})^*) \pi_M.  $$
Applying to both sides of this equality $\id_{M\ot X}\ot \phi_{M,X}$ one gets
\begin{align}\label{dinat-case1}(\id_{M\ot X}\ot \phi_{M,X})  a_{M,X,(M\ot X)^*} \pi_{ M\ot X}= (\id_M\ot coev_X\ot \id_{M^*})\pi_M
\end{align}
We then produce an \'etale algebra  $$\Ac(\Do):=\Ac_{\Do^{\rev}, \ca}=\oint_{X\in \Rep(H)} (S,\beta).$$ This object is a subalgebra of the usual adjoint algebra $\Ac_\ca$. Not to be confused with the adjoint algebra $\Ac_\Do$ of the tensor category $\Do$.

In this Section we shall focus our attention to these particular subalgebras measuring the inclusion $\Do\hookrightarrow \ca.$ First, let us show that there is some sort of exact sequence. This will be more clear when computing some particular examples.

For any $X\in \ca$,  $Y\in \Do$ let $\lambda_Y:\Ac_\Do\to Y\ot Y^*$,  $\pi_X:\Ac_\ca\to X\ot X^*$  be the dinatural transformations associated with these ends. Also,   $\alpha_X:\Ac(\Do) \to X\ot X^*$ denote  the dinatural transformations that satisfy \eqref{dinat-module}. 
\begin{prop}\label{exact-seq-adj} Let $\Do\subseteq \ca$ be a tensor subcategory. There are algebra morphisms 
$$\Ac(\Do) \xrightarrow{ \iota}  \Ac_\ca \xrightarrow{ q}  \Ac_\Do,$$
such that $q\circ \iota=u\circ \alpha_{\uno} $.
\end{prop}
\pf  Since $\Do$ is a subcategory of $\ca$, by the universal property of the end, there exists a map
$$q:\Ac_\ca \to \Ac_\Do$$
such that $\lambda_Y q=\pi_Y$ for any $Y\in \Do$. By the definition of the product of the adjoint algebras given in Section \ref{Section:adjoints}, it follows that $q$ is an algebra map. It follows from Proposition \ref{restriction-sub} that, there exists a map $\iota: \Ac(\Do) \to \Ac_\ca$ such that $\pi_X\iota=\alpha_X $ for any $X\in \ca$. Using Proposition \ref{alg-braid-adj} it follows that $\iota$ is an algebra map. 

Let $Y\in \Do$. Then, it follows from the definition of the unit of $\Ac_\Do$ that
$$\lambda_Y\circ  u\circ \alpha_{\uno} =\coev_Y \circ \alpha_{\uno}. $$
Also
$$\lambda_Y\circ  q \circ \iota=\pi_Y \circ \iota= \alpha_Y.$$
Since $\alpha$ satisfies \eqref{dinat-case1}, it follows that $\alpha_Y= \coev_Y \circ \alpha_{\uno}, $ for any $Y\in \Do$. Whence $q\circ \iota=u\circ \alpha_{\uno} $.
\epf

\subsection{Case $\ca=\Do\boxtimes \Bc$}\label{subs:deligne} Let $\Do$, $\Bc$ be tensor categories. The Deligne tensor product $\ca=\Do\boxtimes \Bc$ has structure of tensor category. The tensor category $\Do$ can be seen as a tensor subcategory of $\ca$ in a canonical way 
$$ \Do \hookrightarrow\Do\boxtimes \Bc, $$
$$X\mapsto X\boxtimes \uno.$$ 
In the next Proposition  we shall make use of the canonical equivalence $$Z(\Do\boxtimes \Bc)\simeq Z(\Do)\boxtimes Z(\Bc).$$
Thus, identifying $ Z(\Bc)$ with the tensor subcategory $\uno \boxtimes Z(\Bc)\hookrightarrow Z(\Do)\boxtimes Z(\Bc)=Z(\ca).$
\begin{prop} There exists an isomorphism of algebras $\Ac(\Do)\simeq \Ac_\Bc$.
\end{prop}
\pf To simplify calculations, we shall assume that both tensor categories are strict. Let $\pi_B:\Ac_\Bc\to B\ot B^*$, $B\in \Bc$, be the dinatural transformations associated with $\Ac_\Bc$. For any $B\in \Bc$, $X\in \Do$ define
$$\widehat{\pi}_{ X\boxtimes B} = (\id_{\uno \boxtimes B}\ot \phi_{X\boxtimes B, {}^*X \boxtimes \uno}  ( \id_{\uno \boxtimes B}\ot ev_X)^*) \pi_B.$$
Note that this definition of $\widehat{\pi}$ is given by using \eqref{dinat-case0}, with trivial associativities. One can prove that $\widehat{\pi}_{ X\boxtimes B}:\Ac_\Bc\to  X\boxtimes B \ot ( X\boxtimes B)^*$ defines a dinatural transformation lifting the dinatural transformation $\pi$ and it satisfies equation \eqref{dinat-case0}. Let be $E\in \Do\boxtimes \Bc$  an object, and $\lambda_A:E\to A\ot A^*$, $A\in \ca$, be a dinatural transformation such that it satisfies \eqref{dinat-case0}. It follows that $\lambda$ is determined by its restriction on elements of the form $\uno \boxtimes B$, $B\in \Bc$, thus defining a dinatural transformation
$$\lambda_{\uno \boxtimes B}:E\to B\ot B^*.$$ Hence, there exists a map $h:E\to \Ac_\Bc$ such that $\pi_B\circ h= \lambda_{\uno \boxtimes B}$. Since both, $\widehat{\pi}$ and $\lambda$ are determined by its values in elements of the form $\uno \boxtimes B$, then $\widehat{\pi}\circ h= \lambda$, proving that $\Ac(\Do)\simeq \Ac_\Bc$.
\epf

\begin{question} If $\Do\to \ca\to \Bc\boxtimes \End(\Mo)$ is an exact sequence of tensor categories relative to a module category $\Mo$, see \cite{EG}, is there a relation between $\Ac(\Do)$ and $\Ac_\Bc$ ?
\end{question}

\subsection{Case when $\ca=\Rep(H)$, $H$ a Hopf algebra}
 
Observe first that $\ca^{\rev}=\Rep(H^{\cop}). $ Tensor subcategories of categories of representations of Hopf algebras are in correspondence with  Hopf algebra quotients. Let $\pi: H\to Q $ be a Hopf algebra projection. In this case, define $\Do=\Rep(Q)$. The inclusion $\Rep(Q)^{\rev} \hookrightarrow \Rep(H)^{\rev}$ is given via $\pi$.
\begin{prop} There exists an isomorphism
$$\Ac(\Do)\simeq \{h\in H: h\_1\ot \pi(h\_2)=h\ot 1\}=H^{co\, \pi}. $$
\end{prop}
\pf
For any $M\in \Rep(H)$ let us define 
$$\lambda_M:H^{co\, \pi}\to M\otk M^*, $$
$$\lambda_M(h)=\sum_i h_i\otk h^i,$$
where $\sum_i  h^i(m)\, h_i=h\cdot m$, for any $h\in (H^{\cop})^{co\, \pi}$, $m\in M$.

\begin{claim} $\lambda$ is a dinatural transformation, and it satisfy \eqref{dinat-case1}.
\end{claim}
\pf[Proof of Claim] Dinaturality of $\lambda$ follows easily. Let $M\in \Rep(H)$, $X\in \Rep(Q)$, $m\in M$, $y\in X$, then the right hand side of  \eqref{dinat-case1} evaluated in $h\in H^{co\, \pi}$ equals
\begin{align*} (\id_M\ot coev_X\ot \id_{M^*})\lambda_M(h)= \sum_{i,j} h_i\ot x_j\ot x^j\ot h^i.
\end{align*}
Here $\{x_i\}, \{x^i\}$ are dual basis if $X$ and $X^*$. When evaluating the right most tensorands in $y\ot m$ of the above equality it gives
$$\sum_{i,j} h_i\ot x_j\ot x^j(y)  h^i(m)=  h\cdot m \ot y.$$
The left hand side of  \eqref{dinat-case1} evaluated in $h\in H^{co\, \pi}$ equals
$$\lambda_{M\ot X}(h)= \sum_i \tilde{h}_i\otk \tilde{h}^i.$$
Evaluating the right tensorand in $m\ot y$ it gives
$$\sum_i \tilde{h}_i\otk \tilde{h}^i(m\ot y)= h\cdot (m\ot y)=h\_1\cdot m\ot \pi(h\_2)\cdot y.$$
Since $h\in H^{co\, \pi}$ it follows that both sides are equal.
\epf
Now, let be $E\in \Rep(H)$ an object equipped with a dinatural transformation $\alpha_M:E\to M\otk M^*$ such that it satisfies \eqref{dinat-case1}. If we set $\alpha_H(e)=\sum_i e_i\ot e^i$,  define
$$\beta:E\to H, \quad \beta(e)=\sum_i e^i(1) e_i.$$
Since $\alpha$ satisfies \eqref{dinat-case1}, it follows that $\beta(E)\subseteq H^{co\, \pi}$, and dinaturality of $\alpha$ implies that $\alpha= \lambda\circ \beta$, proving that $\Ac(\Do)\simeq H^{co\, \pi}$.
\epf

\section{Apendix}

Let $\Mo$ be an  exact indecomposable  left $\ca$-module category, and let $\Do$ be a tensor subcategory of $\ca^*_\Mo$. For any $X\in \ca$ and $(F,\gamma)\in Z_\Do(\End(\Mo))$ the functor
$$ \Hom_{\Mo}(X\triangleright -,F(-)) :\Mo^{\op}\times \Mo \to vect_\ku,$$
has  $\Do$-prebalancing given by
\begin{align}\begin{split}\label{prebalancing-hom11} &\beta^X_{D,M,N}: \Hom_{\Mo}(X\triangleright M,F(D(N))) \to \Hom_{\Mo}(X\triangleright D^{\la}(M),F(N)),\\
&\beta^X_{D,M,N}(f)=(ev_D)_{F(N)} D^{\la}( (\gamma_D)_N f) \tilde{d}^{-1}_{X,M},
\end{split}
\end{align}
for any $(D,d)\in \Do$. Here $ev_D:D^{\la}\circ D \to \Id$ is the evaluation of the adjunction $(D^{\la}, D)$, and $\tilde{d}_{X,M}: D^{\la}(X\triangleright M)\to X\triangleright D^{\la}( M)$ is the $\ca$-module structure of the functor $D^{\la}$. Also, the functor $$  \Hom_{\ca}(X,\uhom(-,F(-))) :\Mo^{\op}\times \Mo \to vect_\ku$$
 has a $\Do$-prebalancing given by
$$b^X_{D,M,N}: \Hom_{\ca}(X,\uhom(M,F(D(N)))) \to \Hom_{\ca}(X,\uhom(D^{\la}(M),F(N))), $$
$$ b^X_{D,M,N} (f)=\beta^D_{M,N} f. $$
Here $\beta^D_{M,N}$ is the $\Do$-prebalancing presented in Lemma \ref{inthomfunct-preb}. Explicitly
$$\beta^D_{M,N}=\xi^D_{M, F(N)} \uhom(\id_M, (\gamma_D)_N),$$
where $\xi^D$ is the natural isomorphism presented in Lemma \ref{int-hom-adjoint}.  The next result is the proof of Claim \ref{claim-thm1}. 
\begin{prop}\label{proofclaim1} Let  $$\phi^X_{M,N}:\Hom_{\ca}(X,\uhom(M,N))\to \Hom_{\Mo}(X\triangleright M,N),$$
 be the natural isomorphisms presented in \eqref{Hom-interno}.  We have that \begin{equation}\label{phi-commut-prebal}
 \beta^X_{D,M,N}  \phi^X_{M,F(D(N))}= \phi^X_{D^{\la}(M),F(N)}  b^X_{D,M,N},
 \end{equation}
  for any $X\in \ca$, $M, N\in \Mo$.
\end{prop}
\begin{proof}  Take $f\in \Hom_{\ca}(X, \uhom(M,F(D(N))))$. Then, applying $ \psi^X_{D^{\la}(M),F(N)}$ on both sides of 
\eqref{phi-commut-prebal} and evaluating in $f$, one obtains that \eqref{phi-commut-prebal} is equivalent to
\begin{equation}\label{phi-commut-prebal2}
\psi^X_{D^{\la}(M),F(N)}\big( \beta^X_{D,M,N} ( \phi^X_{M,F(D(N))}(f))\big)=b^X_{D,M,N}(f).
 \end{equation}

The right hand side of \eqref{phi-commut-prebal2} is equal to
\begin{align*} &=\xi^D_{M, F(N)} \uhom(\id_M, (\gamma_D)_N) f\\
&=  \psi^{Y}_{D^{\la}(M),F(N)}\big(\alpha^{-1}_{Y\triangleright M, F(N)}(\phi^Y_{M,D(F(N))}(\id)) \tilde{d}^{-1}_{Y,M}\big) \uhom(\id_M, (\gamma_D)_{F(N)}) f\\
&= \psi^{Y}_{D^{\la}(M),F(N)}\big(  \alpha^{-1}_{Y\triangleright M, F(N)}(\phi^Y_{M,D(F(N))}(\id))\tilde{d}^{-1}_{Y,M} \\
&( \uhom(\id, (\gamma_D)_{F(N)}) f)\triangleright \id_{D^{\la}(M)} \big)\\
&=  \psi^{X}_{D^{\la}(M),F(N)}\big(\alpha^{-1}_{Y\triangleright M, F(N)}(\phi^Y_{M,D(F(N))}(\id))D^{\la}(\uhom(\id, (\gamma_D)_{F(N)}) f)\triangleright \id)\\  & \tilde{d}^{-1}_{X,M}\big) \\
&= \psi^{X}_{D^{\la}(M),F(N)}\big(\alpha^{-1}_{X\triangleright M, F(N)}(\phi^Y_{M,D(F(N))}(\id)(\uhom(\id, (\gamma_D)_{F(N)}) f)\triangleright \id)  \\  & \tilde{d}^{-1}_{X,M} \big)\\
&=\psi^{X}_{D^{\la}(M),F(N)}\big(\alpha^{-1}_{X\triangleright M, F(N)}(\phi^X_{M,D(F(N))}(\uhom(\id, (\gamma_D)_{F(N)}) f)) \tilde{d}^{-1}_{X,M} \big)
\end{align*}
The second equality follows from the definition of $\xi^D$ given in  \ref{xi-definition}. Here $Y= \uhom(M, D(F(N)))$, and $\tilde{d}$ is the module structure of the functor $ D^{\la}$. The third equality follows from \eqref{comm-psi3}. Here we have used isomorphisms $\alpha$, see Remark \ref{xi-definition}. The fourth equality follows from the naturality of $\tilde{d} $, the fifth one follows from the naturality of $\alpha$, the sixth equality follows from \eqref{comm-phi3}. Using the naturality of $\alpha$, one can see that the last expression is equal to the left hand side of \eqref{phi-commut-prebal2}. 
\end{proof}

\begin{prop}\label{proofclaim2} Let $Q\in \ca$, and $(\mu, \lambda):  L^Q\to K$ be a natural transformation  in $ Z_\Do(\End(\Mo))$.  Then, for any $(D,d)\in \Do$ we have
\begin{align}\label{mu-betas}\begin{split} \uhom((\ev_D)_M, &\id_{K(M)} ) \psi^Q_{M,K(M)}(\mu_M)=\\
&=\xi^D_{D(M),K(M)}\uhom(\id, (\lambda_D)_M) \psi^Q_{D(M),K(D(M))}(\mu_{D(M)}),
\end{split}
\end{align}

\end{prop} 
\pf It follows from \eqref{comm-psi2} that, the left hand side of \eqref{mu-betas} is equal to
$$\psi^Q_{D^{\la}(D(M)),K(M)}( \mu_M (\id_Q\triangleright (\ev_D)_M )).$$
It follows from \eqref{comm-psi1} that, the right hand side of \eqref{mu-betas} is equal to
$$\xi^D_{D(M),K(M)} \psi^Q_{D(M),D(K(M))}((\lambda_D)_M\mu_{D(M)}). $$
Hence, equation \eqref{mu-betas} is equivalent to
\begin{align}\label{mu-betas2} \begin{split}
\mu_M (&\id_Q\triangleright (\ev_D)_M )=\\& =\phi^Q_{D^{\la}(D(M)),K(M)}\big(\xi^D_{D(M),K(M)} \psi^Q_{D(M),D(K(M))}((\lambda_D)_M\mu_{D(M)}) \big) 
\end{split}
\end{align}
It follows from \eqref{comm-phi3} that, the right hand side of \eqref{mu-betas2} is equal to
\begin{align*} &=\phi^{Y}_{D^{\la}(D(M)),K(M)}\big(\xi^D_{D(M),K(M)}\big) \big(\psi^Q_{D(M),D(K(M))}((\lambda_D)_M\mu_{D(M)}) \triangleright\id \big)\\
&= \alpha^{-1}_{Y\triangleright D(M), K(M)}(\phi^Y_{D(M),D(K(M))}(\id)) \tilde{d}^{-1}_{Y,D(M)}\\ &\big(\psi^Q_{D(M),D(K(M))}((\lambda_D)_M\mu_{D(M)}) \triangleright\id \big)\\
&=\alpha^{-1}_{Y\triangleright D(M), K(M)}(\phi^Y_{D(M),D(K(M))}(\id))   D^{\la}(\psi^Q_{D(M),D(K(M))}((\lambda_D)_M\mu_{D(M)}) \triangleright\id)\\
& \tilde{d}^{-1}_{Q,D(M)}\\
&=\alpha^{-1}_{Q\triangleright D(M), K(M)}(\phi^Y_{D(M),D(K(M))}(\id)(\psi^Q_{D(M),D(K(M))}((\lambda_D)_M\mu_{D(M)}) \triangleright\id)) \\
&\tilde{d}^{-1}_{Q,D(M)}\\
&=\alpha^{-1}_{Q\triangleright D(M), K(M)}(\phi^Q_{D(M),D(K(M))}(\psi^Q_{D(M),D(K(M))}((\lambda_D)_M\mu_{D(M)}) )) \tilde{d}^{-1}_{Q,D(M)}\\
&=\alpha^{-1}_{Q\triangleright D(M), K(M)}((\lambda_D)_M\mu_{D(M)}) \tilde{d}^{-1}_{Q,D(M)}\\
&=\alpha^{-1}_{Q\triangleright D(M), K(M)}( D(\mu_M) d^{-1}_{Q,M}) \tilde{d}^{-1}_{Q,D(M)}\\
&=\mu_M \alpha^{-1}_{Q\triangleright D(M), Q\triangleright M}(d^{-1}_{Q,M})\tilde{d}^{-1}_{Q,D(M)}=\mu_M (\id_Q\triangleright (\ev_D)_M ).
\end{align*}
Here $Y=\uhom(D(M),D(K(M)))$. The second equality follows from the definition of the isomorphisms $\xi^D$ given in \eqref{xi-definition}. Here $\tilde{d}$ is the module structure of the functor $D^{\la} $. The third equality follows from the naturality of $\tilde{d}$, the fourth one follows from the naturality of $\alpha$, the fifth equality follows from \eqref{comm-phi3}, the sixth equality follows from the fact that $\mu$ is a module natural transformation. The last equality follows from \eqref{adjointp-mod-funct}. 
\epf

\end{document}